\documentclass[article,12pt]{amsart}

\usepackage{amsmath}
\usepackage{amsfonts}
\usepackage{amssymb}
\usepackage{graphicx}
\usepackage{enumitem}
\usepackage{color}
\usepackage{array}
\usepackage{fullpage}
\usepackage{bbm}

\newtheorem{thm}{Theorem}[section]
\newtheorem{cor}[thm]{Corollary}
\newtheorem{lem}[thm]{Lemma}

\theoremstyle{remark}
\newtheorem*{rem}{Remark}
\newtheorem*{defi}{Definition}

\newcommand{\cB}{\mathcal{B}}
\newcommand{\cF}{\mathcal{F}}
\newcommand{\cN}{\mathcal{N}}
\newcommand{\cX}{\mathcal{X}}

\newcommand{\dC}{\mathbb{C}}
\newcommand{\dE}{\mathbb{E}}
\newcommand{\dP}{\mathbb{P}}
\newcommand{\dR}{\mathbb{R}}
\newcommand{\dV}{\mathbb{V}}
\newcommand{\cov}{\dC\textnormal{ov}}

\newcommand{\de}{\mathrm{e}}

\newcommand{\veps}{\varepsilon}
\renewcommand{\vec}{\textnormal{vec}}
\newcommand{\im}{\textnormal{Im}}
\newcommand{\diag}{\textnormal{diag}}

\newcommand{\hsp}{\hspace{0.5cm}}
\newcommand{\tr}{^{\, T}}
\newcommand{\nm}{\vert\hspace{-0.15ex}\vert\hspace{-0.15ex}\vert}
\newcommand{\nmg}{\Big\vert\hspace{-0.15ex}\Big\vert\hspace{-0.15ex}\Big\vert}
\newcommand{\limn}{\lim_{n\, \rightarrow\, +\infty}}
\newcommand{\cst}{C_{st}}
\newcommand{\ind}{\mathbbm{1}}
\newcommand{\wh}{\widehat}
\newcommand{\whn}{\wh{\theta}_{n}}

\def\leq{\leqslant}
\def\geq{\geqslant}

 \numberwithin{equation}{section}
\allowdisplaybreaks

\begin{document}

\title[Moderate deviations in nearly unstable processes]{Moderate deviations in a class of stable but nearly unstable processes}

\author{Fr\'ed\'eric Pro\"ia}
\address{Laboratoire angevin de recherche en math\'ematiques, LAREMA, UMR 6093, CNRS, UNIV Angers, SFR MathSTIC, 2 Bd Lavoisier, 49045 Angers Cedex 01, France.}
\email{frederic.proia@univ-angers.fr}


\begin{abstract}
We consider a stable but nearly unstable autoregressive process of any order. The bridge between stability and instability is expressed by a time-varying companion matrix $A_{n}$ with spectral radius $\rho(A_{n}) < 1$ satisfying $\rho(A_{n}) \rightarrow 1$. In that framework, we establish a moderate deviation principle for the empirical covariance only relying on the elements of $A_{n}$ through $1-\rho(A_{n})$ and, as a by-product, we establish a moderate deviation principle for the OLS estimator when $\Gamma$, the renormalized asymptotic variance of the process, is invertible. Finally, when $\Gamma$ is singular, we also provide a compromise in the form of a moderate deviation principle for a penalized version of the estimator. Our proofs essentially rely on truncations and deviations of $m_{n}$--dependent sequences, with an unbounded rate $(m_{n})$.
\end{abstract}

\maketitle

\section{Introduction and Assumptions}
\label{SecIntro}

Unit root issues have long been crucial in time series econometrics and have therefore focused a great deal of research studies. This sudden demarcation between stability and instability is responsible for many inference problems in linear time series (see Brockwell and Davis \cite{BrockwellDavis91} for a detailed overview of the linear stochastic processes). The remarkable works of Chan and Wei \cite{ChanWei88} encompass, in a much more general context, the now well-known fact that the least squares estimator is $\sqrt{n}$--consistent with Gaussian behavior when the underlying autoregressive process is stable, whereas it is $n$--consistent with asymmetrical distribution when the process is unstable. This rather abrupt change in the rate of convergence and in the asymptotic distribution certainly motivated the wide range of unit root testing procedures, but it also paved the way for studies based on time-varying coefficients. In a nearly unstable autoregressive process, we do not focus on a parameter $\theta$ satisfying $\vert \theta \vert < 1$ or $\vert \theta \vert = 1$ but, instead, the parameter is considered as a sequence $(\theta_{n})$ such that $\vert \theta_{n} \vert < 1$ and $\vert \theta_{n} \vert \rightarrow 1$ as $n \rightarrow +\infty$. This sample size dependent structure allows a continuity between stability and instability. For example, Phillips and Magdalinos \cite{PhillipsMagdalinos07} treat the case where the coefficient is in a $O(\kappa_{n}^{-1})$ neighborhood of the unit root with $\kappa_{n} = n^{\alpha} = o(n)$. Amongst other results, they prove a central limit theorem for the estimator at the rate $\sqrt{n\, \kappa_{n}}$, thereby making a bridge between the stable rate $\sqrt{n}$ and the unstable rate $n$. In the same vein, let us also mention the work of Chan and Wei \cite{ChanWei87}, natural generalizations like the study of Phillips and Lee \cite{PhillipsLee15} related to vector autoregressions, or the recent unified theory of Buchmann and Chan \cite{BuchmannChan13}, focused on nearly unstable autoregressive processes. Our paper is precisely based on the latter topic, in a sense that will be precised in good time.

\smallskip

Given a parametric generating process, the precision of the estimation is usually  assessed by its rate of convergence and the deviations can be seen as a natural continuation after a central limit theorem or even a law of iterated logarithm.  Roughly speaking, they may be used to estimate the exponential decline of the probability of tail events related to the distance between the estimator and the parameter of interest. We refer to Dembo and Zeitouni \cite{DemboZeitouni98} regarding the mathematical formalization. Since the 1980s, numerous authors have worked on large and/or moderate deviations in a time series context under many and varied hypotheses. Without claiming to be exhaustive, one can mention the studies of Donsker and Varadhan \cite{DonskerVaradhan85} and Bercu \textit{et al.} \cite{BercuGamboaRouault97} on stationary Gaussian processes and quadratic forms, the paper of Worms \cite{Worms99} on Markov chains and regression models and the one of Bercu \cite{Bercu01} on first-order Gaussian stable, unstable and explosive processes. One can also mention the works of Mas and Menneteau \cite{MasMenneteau03} on Hilbertian processes, Djellout \textit{et al.} \cite{DjelloutGuillinWu06} on non-linear functionals of moving average processes, Wu and Zhao \cite{WuZhao08} on stationary non-linear processes, Miao and Shen \cite{MiaoShen09} on general autoregressive processes or, more recently, Bitseki Penda \textit{et al.} \cite{BitsekiDjelloutProia13} on first-order processes with correlated errors. All the references inside may complete this concise list.

\smallskip

In this paper, we investigate the moderate deviations of the estimate in stable but nearly unstable autoregressions. This can be seen as a full generalization of the recent work of Miao, Wang and Yang \cite{MiaoWangYang15}, focused on the univariate case. Our proofs essentially rely on truncations and deviations of $m_{n}$--dependent sequences where the rate $(m_{n})$ is unbounded. The main technical contributions are twofold. On the one hand, expressing the nearly instability directly through the sequence of spectral radii of the companion matrix seems, to the best of our knowledge, a new approach having many advantages. For example the authors of the recent paper \cite{BuchmannChan13} introduce a perturbation in the Jordan canonical form of the model (see Thm. 2.1) which is a powerful idea to deal with the subject of their study, but somehow unnecessarily complex for ours. On the other hand, from a purely technical point of view, unbounded truncations have already been used to get moderate deviations (see \textit{e.g.} \cite{MiaoYang08} and \cite{MiaoWangYang15}), but we will see that the vector case treated here and the specific features of the model cannot be adapted as easily to the existing tools. As a consequence, we need to redevelop a full G\"artner-Ellis reasoning to establish the deviations of our unbounded vector truncations. This quite general strategy might inspire future similar studies.

\smallskip

For a fixed $n \geq 1$, let the process be given for some $p \geq 1$ and $k \in \{ 1, \hdots, n \}$ by
\begin{equation*}
X_{n,\, k} = \sum_{i=1}^{p} \theta_{n,\, i}\, X_{n,\, k-i} + \veps_{k}
\end{equation*}
where $(\veps_{k})_{k}$ is a sequence of zero-mean i.i.d. random variables. In an equivalent way, we can consider the vector expression
\begin{equation}
\label{ModVAR}
\Phi_{n,\, k} = A_{n}\, \Phi_{n,\, k-1} + E_{k}
\end{equation}
where $E_{k} = (\veps_{k}, 0, \hdots, 0)\tr$ is a $p$--vectorial noise, $\Phi_{n,\, k} = (X_{n,\, k}, \hdots, X_{n,\, k-p+1})\tr$ and
\begin{equation}
\label{MatA}
A_{n} = \begin{pmatrix}
\theta_{n,\, 1} & \theta_{n,\, 2} & \hdots & \theta_{n,\, p} \\
& I_{p-1} & & 0
\end{pmatrix}
\end{equation}
is the $p \times p$ companion matrix of the autoregressive process. If $(E_{k})_{k}$ has a finite variance, it is well-known that $(\Phi_{k,\, n})_{k}$ is a second-order stationary process having the causal form
\begin{equation}
\label{CausalForm}
\Phi_{n,\, k} = \sum_{\ell=0}^{+\infty} A_{n}^{\ell}\, E_{k-\ell}
\end{equation}
when $\rho(A_{n}) < 1$, that is, when the largest modulus of its eigenvalues is less than 1 (see \textit{e.g.} Thm. 11.3.1 of \cite{BrockwellDavis91} and the fact that each eigenvalue of $A_{n}$ is the inverse of a zero of the autoregressive polynomial of the process). Since $(\veps_{k})_{k}$ is an i.i.d. sequence, the process is strictly stationary with mean zero and variance given by
\begin{equation}
\label{Var}
\Gamma_{n} = \sigma^2\, \sum_{\ell=0}^{+\infty} A_{n}^{\ell}\, K_{p}\, (A_{n}\tr)^{\ell}
\end{equation}
where, for convenience, we will denote in the whole study
\begin{equation}
\label{MatKU}
K_{p} = \begin{pmatrix}
1 & 0 \\
0 & 0_{p-1}
\end{pmatrix} \hsp \text{and} \hsp U_{p} = \begin{pmatrix}
1 \\
0
\end{pmatrix}
\end{equation}
the $p \times p$ matrix with 1 at the top left and 0 elsewhere, and its first column standing for the first vector of the canonical basis of $\dR^{p}$. As a consequence of the causal expression above, the initial vector $\Phi_{n,\, 0}$ is not arbitrary and has to share the distribution of the process. This also implies the relation
\begin{equation}
\label{RelGammaK}
\Gamma_{n} = A_{n}\, \Gamma_{n}\, A_{n}\tr + \sigma^2\, K_{p}.
\end{equation}
As will be largely developped throughout the study, $\Gamma_{n}$ is finite for all $n \geq 1$ but, as $n$ increases, $\nm \Gamma_{n} \nm \rightarrow +\infty$. The keystone matrix $\Gamma$ obtained after a correct standardization of $\Gamma_{n}$ is the renormalized asymptotic variance of the process. Before we start, we define a matrix that will also prove to be crucial to our results,
\begin{equation}
\label{MatB}
B_{n} = I_{p^2} - A_{n} \otimes A_{n}.
\end{equation}
We are now going to introduce and comment the hypotheses that will be needed, though not always simultaneously, in the whole paper. Section \ref{SecMainRes} is devoted to our main results : two statements related to the moderate deviations of the empirical covariance and the OLS estimator, a set of explicit examples and some additional comments and conclusions. Finally, in Section \ref{SecPro} divided into numerous subsections, we will prove all our results, step by step.

\begin{rem}
We denote by $\Vert \cdot \Vert$ the Euclidean vector norm and by $\nm \cdot \nm$ the spectral matrix norm. Other norms may be used, in which case an appropriated subscript is added. Moreover, we will always denote by $\langle \cdot,\cdot \rangle$ the usual inner product of the Euclidean space $\dR^{d}$ for any $d \geq 1$. We write $M^{\, \dagger}$ for the Moore-Penrose pseudo-inverse of any matrix $M$, whose definition and properties may be found in Sec. 0 of \cite{HiriartLemarechal12}.
\end{rem}

\subsection{Hypotheses}
\label{SecHyp}

First of all, we present the hypotheses that we retain.
\begin{enumerate}[label=(H$_\arabic*$)]
\item \label{HypGI} \textit{Gaussian integrability condition}. There exists $\alpha > 0$ such that
\begin{equation*}
\dE\big[ \de^{\alpha\, \veps_1^{\, 2}} \big] < +\infty
\end{equation*}
where $\veps_1$ represents the zero-mean i.i.d. sequence $(\veps_{k})_{k}$ of variance $\sigma^2 > 0$ and fourth-order moment $\tau^4 > 0$.
\item \label{HypCA} \textit{Convergence of the companion matrix}. There exists a $p \times p$ matrix $A$ such that
\begin{equation*}
\limn A_{n} = A
\end{equation*}
with distinct eigenvalues $0 < \vert \lambda_{p} \vert \leq \hdots \leq \vert \lambda_{1} \vert = \rho(A)$, and the top right element of $A$ is non-zero.
\item \label{HypSR} \textit{Spectral radius of the companion matrix}. For all $n \geq 1$, $\rho(A_{n}) < 1$. In addition,
\begin{equation*}
\limn \rho(A_{n}) = \rho(A) = 1.
\end{equation*}
\item \label{HypRC} \textit{Renormalization}. We have the convergences
\begin{equation*}
\limn \frac{B^{-1}_{n}}{\nm B^{-1}_{n} \nm_{*}} = H \hsp \text{and} \hsp \limn (1 - \rho(A_{n}))\, \nm B^{-1}_{n} \nm_{*} = h
\end{equation*}
for some matrix norm, where $H$ is a $p^2 \times p^2$ non-zero matrix and $h > 0$.
\item \label{HypMD} \textit{Moderate deviations}. The moderate deviations scale $(b_{n})$ satisfies
\begin{equation*}
\limn b_{n} = +\infty \hsp \text{and} \hsp \limn \frac{\sqrt{n}\, (1 - \rho(A_{n}))^{\frac{3}{2} + \eta}}{b_{n}} = +\infty
\end{equation*}
for a small $\eta > 0$.
\end{enumerate}

\subsection{Comments on the hypotheses}
\label{SecComm}

First, conceding in \ref{HypCA} that the limiting matrix has distinct eigenvalues is a matter of simplication of the reasonings. Indeed, $A_{n}$ turns out to be diagonalizable for a sufficiently large $n$, and, as a companion matrix, it is well-known that the change of basis is done \textit{via} a Vandermonde matrix having numerous nice properties (more details are given in Section \ref{SecLinAlg}, and a discussion on the case of multiple eigenvalues is provided in Section \ref{SecMultVP}). The top right element of $A_{n}$ is $\theta_{n,\, p}$. So, assuming in \ref{HypCA} that $\theta_{n,\, p} \nrightarrow 0$ ensures that the limit process is still of order $p$ and that 0 cannot be an eigenvalue of $A$, since $\det(A) = (-1)^{p+1}\, \theta_{p}$. Moreover, note that, in \ref{HypRC}, the invertibility of $B_{n}$ for all $n$ is guaranteed by \ref{HypSR}. Indeed, $\rho(A_{n} \otimes A_{n}) = \rho^2(A_{n}) < 1$ (see \textit{e.g.} Lem. 5.6.10 and Cor. 5.6.16 of \cite{HornJohnson92}). In addition, we obviously have, for all $\ell \geq 0$,
\begin{equation*}
\rho(A_{n}^{\ell}) = \rho^{\ell}(A_{n})\, \leq\, \nm A_{n}^{\ell} \nm
\end{equation*}
so that we get
\begin{equation}
\label{Ln}
\frac{1}{1 - \rho(A_{n})}\, \leq\, \sum_{\ell=0}^{+\infty} \nm A_{n}^{\ell} \nm = L_{n}
\end{equation}
giving a lower bound for $L_{n}$. Similarly,
\begin{equation}
\label{Mn}
\frac{1}{(1 - \rho(A_{n}))^2}\, \leq\, \sum_{\ell=0}^{+\infty} (\ell+1)\, \nm A_{n}^{\ell} \nm = M_{n}.
\end{equation}
However, an exact upper bound for these sums may be difficult to reach and may require stringent conditions on the elements of $A_{n}$. We refer the reader to Lemma \ref{LemEquivLn} where, under \ref{HypCA} and \ref{HypSR}, some asymptotic upper bounds are established. We also refer to Section \ref{SecEx} where the explicit calculations in terms of some examples shall help to understand the rates involved in the hypotheses. Now for a fixed $n \geq 1$, let
\begin{equation*}
\mu_{n} = \rho(A_{n}) + \frac{1-\rho(A_{n})}{2} = \frac{\rho(A_{n}) + 1}{2}.
\end{equation*}
Clearly, $\rho(A_{n}) < \mu_{n} < 1$. Hence, according to Prop. 2.3.15 of \cite{Duflo97}, for all $n \geq 0$, there exists a constant $c_{n} > 0$ such that, for all $\ell \geq 0$, $\nm A_{n}^{\ell} \nm\, \leq\, c_{n}\, \mu_{n}^{\ell}$ so that
\begin{equation*}
L_{n}\, \leq\, \frac{c_{n}}{1 - \mu_{n}} < +\infty \hsp \text{and} \hsp M_{n}\, \leq\, \frac{c_{n}}{(1 - \mu_{n})^2} < +\infty.
\end{equation*}
Letting $n$ tend to infinity, it follows from \ref{HypSR} and \ref{HypRC} that
\begin{equation}
\label{LimBL}
\limn \nm B^{-1}_{n} \nm = \limn L_{n} = \limn M_{n} = +\infty.
\end{equation}
Finally, it will be established in good time that there is a limiting matrix $\Gamma$ such that
\begin{equation}
\label{LimGamGen}
\limn \frac{\Gamma_{n}}{\nm B_{n}^{-1} \nm_{*}} = \Gamma
\end{equation}
where $\nm \cdot \nm_{*}$ is the matrix norm of \ref{HypRC}.

\begin{rem}
To facilitate the reading, we consider from now on that the matrix norm $\nm \cdot \nm_{*}$ is identified in \ref{HypRC}, and we will only note $\nm \cdot \nm$ in what follows.
\end{rem}

\section{Main results}
\label{SecMainRes}

This section contains two statements that constitute the main results of the paper. The first of them is quite long to establish and will need numerous technical lemmas, but the second one will essentially be deduced as a corollary of the first one. Subsequently, we provide some explicit examples for a better understanding and an easier interpretation of the hypotheses together with some graphics showing the evolution of the processes and the estimation of the autoregressive parameter. At the end of the section, we discuss the case of multiple eigenvalues. But, first, let us recall the definition of the large and moderate deviation principles (see Sec. 1.2 of \cite{DemboZeitouni98} for more details). In what follows, a speed is considered as a positive sequence increasing to infinity.

\begin{defi}
A sequence of random variables $(U_{n})_{n}$ on a topological space $(\cX, \cB)$ satisfies a large deviation principle (LDP) with speed $(a_{n})$ and rate $I$ if there is a lower semicontinuous mapping $I : \cX \rightarrow \bar{\dR}^{+}$ such that :
\begin{itemize}
\item for any closed set $F \in \cB$,
\begin{equation*}
\limsup_{n\, \rightarrow\, +\infty} ~ \frac{1}{a_{n}} \ln \dP(U_{n} \in F) ~ \leq ~ -\inf_{x\, \in\, F} I(x),
\end{equation*}
\item for any open set $G \in \cB$,
\begin{equation*}
-\inf_{x\, \in\, G} ~ I(x) ~ \leq ~ \liminf_{n\, \rightarrow\, +\infty} ~ \frac{1}{a_{n}} \ln \dP(U_{n} \in G).
\end{equation*}
\end{itemize}
In particular, if the infimum of $I$ coincides on the interior $H^{\circ}$ and the closure $\bar{H}$ of some $H \in \cB$, then
\begin{equation*}
\lim_{n\, \rightarrow\, +\infty} ~ \frac{1}{a_{n}} \ln \dP(U_{n} \in H) = -\inf_{x\, \in\, H} ~ I(x).
\end{equation*}
\end{defi}

\begin{defi}
A sequence of random variables $(V_{n})_{n}$ on a topological space $(\cX, \cB)$ satisfies a moderate deviation principle (MDP) with speed $(b_{n}^{\, 2})$ and rate $I$ if there is a speed $(v_{n})$ with $\frac{v_{n}}{b_{n}} \rightarrow +\infty$ such that $(\frac{v_{n}}{b_{n}}\, V_{n})_{n}$ satisfies a large deviation principle of speed $(b_{n}^{\, 2})$ and rate $I$. 
\end{defi}

\subsection{Moderate deviations}
\label{SecMDP}

We now consider an observable trajectory $X_{n,\, -p+1}, \hdots, X_{n,\, n}$ for some fixed $n \geq 1$, and use it to provide an estimation of the parameter. It is well-known that the ordinary least squares (OLS) estimator of $\theta_{n} = (\theta_{n,\, 1}, \hdots, \theta_{n,\, p})\tr$ is given by
\begin{equation}
\label{OLS}
\whn = S_{n-1}^{\, -1} \sum_{k=1}^{n} \Phi_{n,\, k-1}\, X_{n,\, k} \hsp \text{where} \hsp S_{n-1} = \sum_{k=1}^{n} \Phi_{n,\, k-1}\, \Phi_{n,\, k-1}\tr.
\end{equation}
The first result is dedicated to the empirical variance $\frac{S_{n}}{n}$.

\begin{thm}
\label{ThmMDPCov}
Under hypotheses \ref{HypGI}--\ref{HypMD}, the sequence
\begin{equation*}
\left( \frac{\sqrt{n}\, (1 - \rho(A_{n}))^{\frac{3}{2}}}{b_{n}}\, \vec\bigg( \frac{1}{n} \sum_{k=1}^{n} (\Phi_{n,\, k}\, \Phi_{n,\, k}\tr - \Gamma_{n}) \bigg) \right)_{\!n\, \geq\, 1}
\end{equation*}
satisfies an LDP with speed $(b_{n}^{\, 2})$ and a rate function $I_{\Gamma} : \dR^{p^2} \rightarrow \bar{\dR}^{+}$ defined as
\begin{equation*}
I_{\Gamma}(x) = \left\{
\begin{array}{ll}
\frac{1}{2\, h^3}\, \langle x, \Upsilon^{\, \dagger}\, x \rangle & \mbox{for } x \in \im(\Upsilon) \\
+\infty & \mbox{otherwise}
\end{array}
\right.
\end{equation*}
where $\Upsilon$ is explicitely given in \eqref{LimVarRenorm} and $h$ comes from \ref{HypRC}.
\end{thm}
\begin{proof}
See Section \ref{SecProMDPCov}.
\end{proof}

\begin{rem}
Through vectorization, this MDP is established on $\dR^{p^2}$ in order to avoid any confusion in the notations, but we might work in $\dR^{p \times p}$ as well. The associated rate function would only require a slight modification of the proof.
\end{rem}

\begin{rem}
To be punctilious, we may add a small $\epsilon > 0$ to the diagonal of $S_{n-1}$ to ensure that it is non-sigular for all $n \geq 1$ without disturbing the asymptotic behavior.
\end{rem}

When the variance $\Gamma$ given in \eqref{LimGamGen} is invertible, we establish the MDP for the OLS in the theorem that follows. However, when it is not the case, there are some technical complications and, to reach an intermediate result, we need to introduce a penalized version of the OLS. For a small $\pi \geq 0$, define
\begin{equation}
\label{OLSPen}
\whn^{\, \pi} = (S_{n-1}^{\, \pi})^{-1} \sum_{k=1}^{n} \Phi_{n,\, k-1}\, X_{n,\, k} \hsp \text{where} \hsp S_{n-1}^{\, \pi} = S_{n-1} + \pi\, n\, \nm B_{n}^{-1} \nm\, I_{p}
\end{equation}
with possibly $\pi = 0$ if $\Gamma$ is invertible, in which case it is clearly the standard OLS given above, but necessarily $\pi > 0$ otherwise. Consider also the penalized version of the variance and the corrected parameter
\begin{equation}
\label{VarPen}
\Gamma_{\pi} = \Gamma + \pi\, I_{p} \hsp \text{and} \hsp \theta_{n}^{\, \pi} = (S_{n-1}^{\, \pi})^{-1} S_{n-1}\, \theta_{n}.
\end{equation}
By construction, $\Gamma$ is, at worst, non-negative definite and for $\pi > 0$, $\Gamma_{\pi}$ turns out to be invertible. The same goes for $S_{n-1}^{\, \pi}$.

\begin{cor}
\label{CorMDPOLS}
Under hypotheses \ref{HypGI}--\ref{HypMD}, for all $\pi > 0$, the sequence
\begin{equation*}
\left( \frac{\sqrt{n}}{b_{n}\, (1 - \rho(A_{n}))^{\frac{1}{2}}}\, \big( \whn^{\, \pi} - \theta_{n}^{\, \pi} \big) \right)_{\!n\, \geq\, 1}
\end{equation*}
satisfies an LDP with speed $(b_{n}^{\, 2})$ and a rate function $I_{\theta}^{\, \pi} : \dR^{p} \rightarrow \bar{\dR}^{+}$ defined as
\begin{equation*}
I_{\theta}^{\, \pi}(x) = \left\{
\begin{array}{ll}
\frac{h}{2\, \sigma^2}\, \langle x, \Gamma_{\pi}\, \Gamma^{\, \dagger}\, \Gamma_{\pi}\, x \rangle & \mbox{for } x \in \im(\Gamma_{\pi}^{-1}\, \Gamma) \\
+\infty & \mbox{otherwise}
\end{array}
\right.
\end{equation*}
where the variance $\Gamma$ is given in \eqref{LimGamGen}, $\Gamma_{\pi}$ is the penalized variance given in \eqref{VarPen} and $h$ comes from \ref{HypRC}, respectively. If in addition $\Gamma$ is invertible, then the sequence
\begin{equation*}
\left( \frac{\sqrt{n}}{b_{n}\, (1 - \rho(A_{n}))^{\frac{1}{2}}}\, \big( \whn - \theta_{n} \big) \right)_{\!n\, \geq\, 1}
\end{equation*}
satisfies an LDP with speed $(b_{n}^{\, 2})$ and a rate function $I_{\theta} : \dR^{p} \rightarrow \dR^{+}$ defined as
\begin{equation*}
I_{\theta}(x) = \frac{h}{2\, \sigma^2}\, \langle x, \Gamma\, x \rangle.
\end{equation*}
\end{cor}
\begin{proof}
See Section \ref{SecProMDPOLS}.
\end{proof}

To sum up, this result shows that, when $\Gamma$ is invertible, the OLS satisfies an MDP, and even when $\Gamma$ is singular, one may reach a compromise by getting an MDP for a penalized estimation. In the same vein, notice also that, in the invertible case,
\begin{equation*}
\lim_{\pi\, \rightarrow\, 0^{+}} I_{\theta}^{\, \pi}(x) = I_{\theta}(x).
\end{equation*}

\begin{rem}
In the stable case where $\rho(A_{n}) = \rho(A) < 1$, we simply have $(1 - \rho(A_{n}))\, \nm B^{-1}_{n} \nm = h$ and $\Gamma_{n}\, \nm B_{n}^{-1} \nm^{-1} = \Gamma$ for all $n \geq 1$. By contraction, the MDP of Corollary \ref{CorMDPOLS} coincides with the one of Thm. 3 of \cite{Worms99} when $\Gamma$ is invertible.
\end{rem}

\subsection{Some explicit examples}
\label{SecEx}

Before giving some examples, we can already note that \ref{HypMD} implies $\sqrt{n}\, (1 - \rho(A_{n})) \rightarrow +\infty$. Thus, necessarily, the convergence $1-\rho(A_{n}) \rightarrow 0$ cannot occur with an exponential rate, this is the reason why we focus on polynomial rates of the form $1-\rho(A_{n}) = c\, n^{-\alpha}$ for some $c > 0$ in this section. Accordingly, in all the examples below, \ref{HypMD} is only possible when $0 < \alpha < \frac{1}{3 + 2 \eta} < \frac{1}{3}$. Thus, one cannot expect a sequence of coefficients moving too fast toward instability. The domain of validify of the speed of the MDP will be
\begin{equation*}
1\, \ll\, b_{n}\, \ll\, n^{\, \frac{1 - (3 + 2 \eta) \alpha}{2}}\, \ll\, \sqrt{n}.
\end{equation*}

\subsubsection{Univariate case with one nearly unit root}
\label{SecEx1}

Suppose that $p=1$. Then, \ref{HypCA} and \ref{HypSR} imply that $\vert \theta_{n} \vert < 1$ and $\theta_{n} \rightarrow \pm 1$. We also have $B_{n} = 1 - \theta_{n}^{\, 2}$ and \ref{HypRC} can be expressed like
\begin{equation*}
\limn \frac{B^{-1}_{n}}{\vert B^{-1}_{n} \vert} = 1 \hsp \text{and} \hsp \limn (1 - \vert \theta_{n} \vert)\, \vert B^{-1}_{n} \vert = \frac{1}{2}.
\end{equation*}
A straightforward calculation shows that
\begin{equation*}
\Gamma_{n} = \frac{\sigma^2}{1 - \theta_{n}^{\, 2}} \hsp \text{and} \hsp \Gamma = \sigma^2 > 0
\end{equation*}
so that we can choose $\pi=0$. The standard cases, illustrated on Figure \ref{FigEx1}, are $\theta_{n} = 1 - c_1\, n^{-\alpha}$ for the positive unit root and $\theta_{n} = -1 + c_2\, n^{-\alpha}$ for the negative unit root, with $c_1, c_2 > 0$ and $\alpha > 0$. The rate function associated with Corollary \ref{CorMDPOLS} is $I_{\theta}(x) = \frac{x^2}{4}$, which corresponds to Prop. 2.1 of \cite{MiaoWangYang15}. Indeed, their rate $x \mapsto \frac{x^2}{2}$ is associated to an LDP with the renormalization $(1 - \theta_{n}^{\, 2})^{\frac{1}{2}}$ whereas our normalization is $(1 - \vert \theta_{n} \vert)^{\frac{1}{2}}$. By contraction, the asymptotic factor $\sqrt{2}$ explains the difference.

\begin{figure}[h]
\includegraphics[scale=0.6]{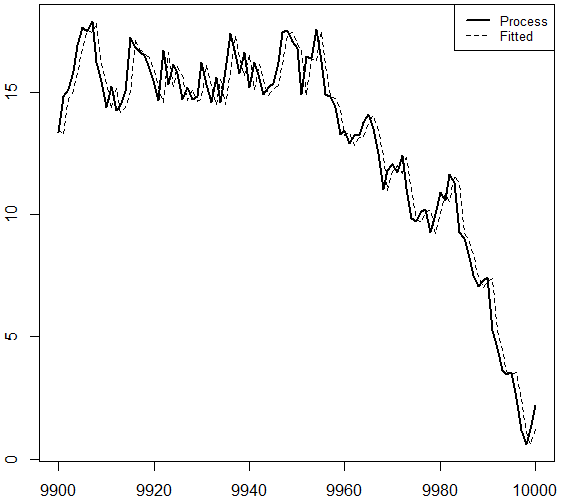} \hsp \includegraphics[scale=0.6]{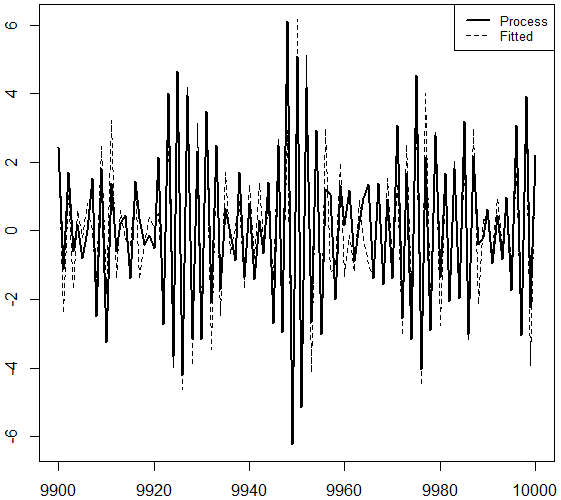}
\caption{Simulation of the process (solid line) and fitted values (dotted line) for $n=10^4$, $\pi = 0$ and $\cN(0,1)$ innovations. The setting is $c_1 = 0.1$ and $\alpha = 0.32$ on the left, $c_2 = 0.1$ and $\alpha = 0.32$ on the right.}
\label{FigEx1}
\end{figure}

\subsubsection{Bivariate case with one nearly unit root}
\label{SecEx2}

Suppose now that $p=2$ and $\textnormal{sp}(A) = \{ \pm 1, \lambda \}$ with $\vert \lambda \vert < 1$. This situation occurs, for example, when
\begin{equation*}
A_{n} = \begin{pmatrix}
\lambda + 1 - c\, n^{-\alpha} & -\lambda\, (1 - c\, n^{-\alpha}) \\
1 & 0
\end{pmatrix}
\end{equation*}
whose eigenvalues are $1 - c\, n^{-\alpha}$ and $\lambda$. This is illustrated on Figure \ref{FigEx2}. For $c > 0$ and $\alpha > 0$, \ref{HypCA} and \ref{HypSR} are satisfied. The direct calculation gives
\begin{equation*}
B_{n}^{-1} = \frac{1}{2\, c\, (\lambda - 1)^2}\, \begin{pmatrix}
1 & -\lambda & -\lambda & \lambda^2 \\
1 & -\lambda & -\lambda & \lambda^2 \\
1 & -\lambda & -\lambda & \lambda^2 \\
1 & -\lambda & -\lambda & \lambda^2
\end{pmatrix} \big( n^{\alpha}\ + O(1) \big)
\end{equation*}
whence we obtain
\begin{equation*}
\limn \frac{B^{-1}_{n}}{\nm B^{-1}_{n} \nm_{1}} = \frac{1}{4} \begin{pmatrix}
1 & -\lambda & -\lambda & \lambda^2 \\
1 & -\lambda & -\lambda & \lambda^2 \\
1 & -\lambda & -\lambda & \lambda^2 \\
1 & -\lambda & -\lambda & \lambda^2
\end{pmatrix} \hsp \text{and} \hsp \limn n^{-\alpha}\, \nm B^{-1}_{n} \nm_{1} = \frac{2}{c\, (\lambda - 1)^2}
\end{equation*}
so \ref{HypRC} is satisfied with the 1--norm. The choice $\pi=0$ is impossible, and we finally find
\begin{equation*}
\Gamma_{\pi} = \frac{\sigma^2}{4} \begin{pmatrix}
1 + \frac{4}{\sigma^2}\, \pi & 1 \\
1 & 1 + \frac{4}{\sigma^2}\, \pi
\end{pmatrix}.
\end{equation*}

\begin{figure}[h]
\includegraphics[scale=0.6]{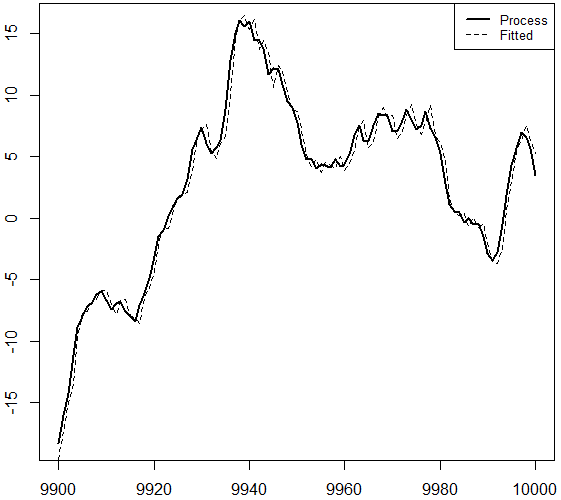} \hsp \includegraphics[scale=0.6]{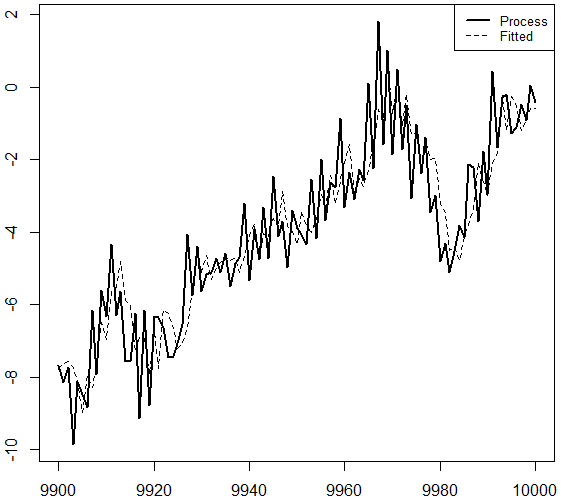}
\caption{Simulation of the process (solid line) and fitted values (dotted line) for $n=10^4$, $\pi = 10^{-5}$ and $\cN(0,1)$ innovations. The setting is $\lambda = 0.5$, $c = 0.1$ and $\alpha = 0.32$ on the left, $\lambda = -0.67$, $c = 0.2$ and $\alpha = 0.25$ on the right.}
\label{FigEx2}
\end{figure}

\subsubsection{Bivariate case with two nearly unit roots}
\label{SecEx3}

Following the same lines, suppose that $p=2$ and $\textnormal{sp}(A) = \{ -1,1 \}$. This situation occurs, for example, when
\begin{equation*}
A_{n} = \begin{pmatrix}
(c_2 - c_1) n^{-\alpha} & (1-c_1\, n^{-\alpha})(1-c_2\, n^{-\alpha}) \\
1 & 0
\end{pmatrix}
\end{equation*}
whose eigenvalues are $1 - c_1\, n^{-\alpha}$ and $-1 + c_2\, n^{-\alpha}$. This is illustrated on Figure \ref{FigEx3}. For $c_1, c_2 > 0$ and $\alpha > 0$, \ref{HypCA} and \ref{HypSR} are satisfied. The direct calculation gives
\begin{equation*}
B_{n}^{-1} = \frac{1}{8\, c_1\, c_2}\, \begin{pmatrix}
c_1+c_2 & c_2-c_1 & c_2-c_1 & c_1+c_2 \\
c_2-c_1 & c_1+c_2 & c_1+c_2 & c_2-c_1 \\
c_2-c_1 & c_1+c_2 & c_1+c_2 & c_2-c_1 \\
c_1+c_2 & c_2-c_1 & c_2-c_1 & c_1+c_2
\end{pmatrix} \big( n^{\alpha}\ + O(1) \big)
\end{equation*}
whence we obtain
\begin{equation*}
\limn \frac{B^{-1}_{n}}{\nm B^{-1}_{n} \nm_{1}} = \frac{1}{2\, (c_1+c_2) + 2\, \vert c_2 - c_1 \vert} \begin{pmatrix}
c_1+c_2 & c_2-c_1 & c_2-c_1 & c_1+c_2 \\
c_2-c_1 & c_1+c_2 & c_1+c_2 & c_2-c_1 \\
c_2-c_1 & c_1+c_2 & c_1+c_2 & c_2-c_1 \\
c_1+c_2 & c_2-c_1 & c_2-c_1 & c_1+c_2
\end{pmatrix}.
\end{equation*}
Moreover,
\begin{equation*}
\limn n^{-\alpha}\, \nm B^{-1}_{n} \nm_{1} = \frac{(c_1+c_2) + \vert c_2 - c_1 \vert}{4\, c_1\, c_2}
\end{equation*}
so \ref{HypRC} is satisfied with the 1--norm. The choice $\pi=0$ is possible and we finally find
\begin{equation*}
\Gamma = \frac{\sigma^2}{2\, (c_1+c_2) + 2\, \vert c_2 - c_1 \vert} \begin{pmatrix}
c_1+c_2 & c_2-c_1 \\
c_2-c_1 & c_1+c_2
\end{pmatrix}.
\end{equation*}

\begin{figure}[h]
\includegraphics[scale=0.6]{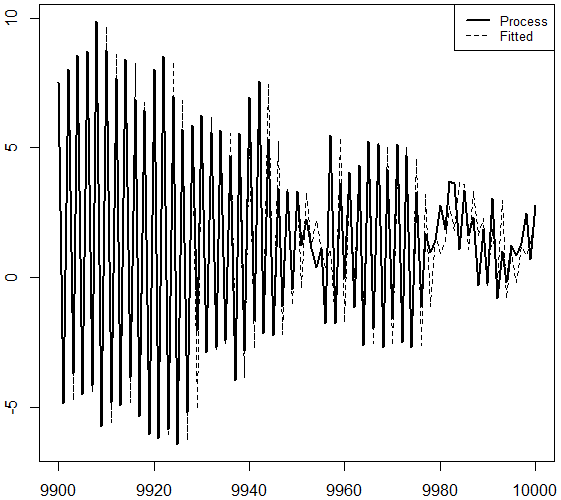} \hsp \includegraphics[scale=0.6]{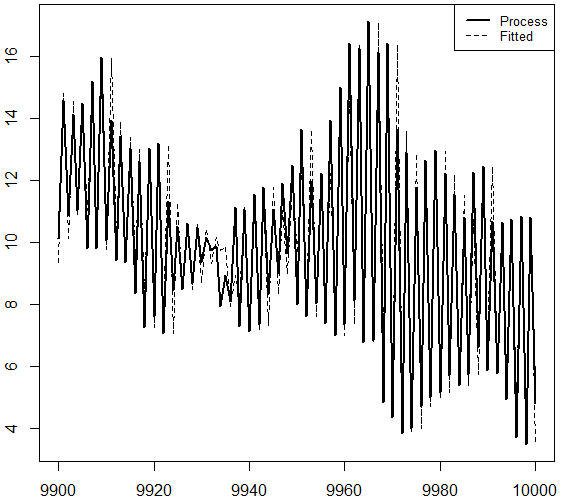}
\caption{Simulation of the process (solid line) and fitted values (dotted line) for $n=10^4$, $\pi = 0$ and $\cN(0,1)$ innovations. The setting is $c_1 = 0.1$, $c_2 = 0.2$ and $\alpha = 0.32$ on the left, $c_1 = 0.01$, $c_2 = 0.01$ and $\alpha = 0.25$ on the right.}
\label{FigEx3}
\end{figure}

\subsection{Discussion on multiple eigenvalues and conclusion}
\label{SecMultVP}

As we will see in the proof of Lemma \ref{LemEquivLn}, the distinct eigenvalues assumption \ref{HypCA} is sufficient to reach our results. However, a less stringent formulation of \ref{HypCA} could be :
\begin{enumerate}[label=(H$^{\, \prime}_\arabic*$)]
\setcounter{enumi}{1}
\item \label{HypCAb} \textit{Convergence of the companion matrix}. There exists a $p \times p$ matrix $A$ such that
\begin{equation*}
\limn A_{n} = A
\end{equation*}
and the top right element of $A$ is non-zero. In addition, there exists a rank $n_0$ such that, for all $n > n_0$, $A_{n}$ is diagonalizable and the change of basis matrix $P_{n}$ satisfies $\nm P_{n} \nm \leq \cst$ and $\nm P_{n}^{\, -1} \nm \leq \cst$.
\end{enumerate}
In general, multiple eigenvalues may not falsify our reasonings, except when the multiplicity concerns the eigenvalues whose modulus tends to 1. Indeed, the coefficients of $\nm A_{n}^{\ell} \nm$ may grow faster in that case. Consider the simple bivariate example where
\begin{equation*}
A_{n}^{\ell} ~ = ~ \begin{pmatrix}
a_{11,\, \ell} & a_{12,\, \ell} \\
a_{21,\, \ell} & a_{22,\, \ell}
\end{pmatrix} ~ = ~ \begin{pmatrix}
\theta_{n,\, 1}\, a_{11,\, \ell-1} + \theta_{n,\, 2}\, a_{11,\, \ell-2} & \theta_{n,\, 2}\, a_{11,\, \ell-1} \\
\theta_{n,\, 1}\, a_{21,\, \ell-1} + \theta_{n,\, 2}\, a_{21,\, \ell-2} & \theta_{n,\, 2}\, a_{21,\, \ell-1}
\end{pmatrix}.
\end{equation*}
Then, it is not hard to solve this linear difference equation whose characteristic roots are the eigenvalues of $A_{n}$. In case of multiplicity, the top left term takes the form of
\begin{equation*}
a_{11,\, \ell} = (c_{n} + d_{n}\, \ell)\, \rho^{\ell}(A_{n})
\end{equation*}
and even if $\vert c_{n} \vert \leq \cst$ and $\vert d_{n} \vert \leq \cst$ for $n$ large enough, it follows that
\begin{equation*}
\sum_{\ell=0}^{+\infty} \nm A_{n}^{\ell} \nm \sim \frac{\cst}{(1 - \rho(A_{n}))^2}.
\end{equation*}
That invalidates all our reasonings and, in that case, new approaches are needed to potentially reach the moderate deviations. From our viewpoint, this is the main weakness of the set of hypotheses. As it is already observed in \cite{ChanWei88}, multiple unit roots located at 1 influence the rate of convergence of the OLS. We conjecture that the same phenomenon occurs here and that a larger power should come with $1-\rho(A_{n})$ in the renormalization.

\smallskip

To sum up, this study is a wide generalization of \cite{MiaoWangYang15} and, although not complete in virtue of the latter remark, it covers most of the MDP issues for the estimation in the stable but nearly unstable case. Large deviations would undoubtedly be a very useful and challenging study to carry out, naturally extending this one. However, to the best of our knowledge, it is not even entirely treated in the stable time-invariant case $\rho(A_{n}) = \rho(A) < 1$, clearly revealing the complexity of the problem. A complicated but stimulating trail for future studies could rely on the exponential, and not only polynomial, neighborhood of the unit root. Along the same lines and even if it is of less practical interest, we might as well focus on the explosive side of the unit roots, where new theoretical developments are necessary.

\section{Technical proofs}
\label{SecPro}

In all the proofs, $\cst$ denotes a generic positive constant that is not necessarily identical from one line to another. We will frequently use the fact that $\Vert \vec(\cdot) \Vert = \nm \cdot \nm_{F} \leq \cst\, \nm \cdot \nm$. For asymptotic equivalences, $f_{n} \asymp g_{n}$ means that both $f_{n} = O(g_{n})$ and $g_{n} = O(f_{n})$ whereas $f_{n} \sim g_{n}$ stands for $\frac{f_{n}}{g_{n}} \rightarrow 1$.

\subsection{Some linear algebra tools}
\label{SecLinAlg}

Thereafter, we denote by $\lambda_{1}, \hdots, \lambda_{p}$ the (distinct) eigenvalues of $A$ and $\lambda_{n,\, 1}, \hdots, \lambda_{n,\, p}$ those of $A_{n}$, in descending order of modulus. We start by establishing two lemmas that will prove to be very useful in what follows.

\begin{lem}
\label{LemEquivLn}
Under hypotheses \ref{HypCA} and \ref{HypSR}, as $n$ tends to infinity,
\begin{equation*}
\sum_{\ell=0}^{+\infty} \nm A_{n}^{\ell} \nm\, \asymp\, \frac{1}{1 - \rho(A_{n})} \hsp \text{and} \hsp \sum_{\ell=0}^{+\infty} (\ell+1)\, \nm A_{n}^{\ell} \nm\, \asymp\, \frac{1}{(1 - \rho(A_{n}))^2}.
\end{equation*}
\end{lem}
\begin{proof}
The lower bounds are established in Section \ref{SecComm}, in \eqref{Ln} and \eqref{Mn} precisely. For the upper bounds, fix
\begin{equation*}
\delta = \frac{2}{\vert \lambda_{p} \vert}, \hsp \epsilon_1 = \frac{1}{2}\, \min_{\substack{1\, \leq\, i,j\, \leq\, p \\ i\, \neq\, j}} \bigg\vert \frac{1}{\lambda_{i}} - \frac{1}{\lambda_{j}} \bigg\vert \hsp \text{and} \hsp \epsilon_2 = 2\, \max_{\substack{1\, \leq\, i,j\, \leq\, p \\ i\, \neq\, j}} \bigg\vert \frac{1}{\lambda_{i}} - \frac{1}{\lambda_{j}} \bigg\vert.
\end{equation*}
According to Thm. 2.4.9.2 of \cite{HornJohnson92}, \ref{HypCA} implies the existence of a rank $n_0 = n_0(\delta, \epsilon_1, \epsilon_2)$ such that, for all $n > n_0$, the eigenvalues of $A_{n}$ satisfy
\begin{equation}
\label{InegInvVP}
0 < \max_{1\, \leq\, i\, \leq\, p} \bigg\vert \frac{1}{\lambda_{n,\, i}} \bigg\vert < \delta
\end{equation}
and
\begin{equation}
\label{InegDiffInvVP}
\epsilon_1 < \min_{\substack{1\, \leq\, i,j\, \leq\, p \\ i\, \neq\, j}} \bigg\vert \frac{1}{\lambda_{n,\, i}} - \frac{1}{\lambda_{n,\, j}} \bigg\vert < \max_{\substack{1\, \leq\, i,j\, \leq\, p \\ i\, \neq\, j}} \bigg\vert \frac{1}{\lambda_{n,\, i}} - \frac{1}{\lambda_{n,\, j}} \bigg\vert < \epsilon_2.
\end{equation}
Let $P_{n}$ be a change of basis matrix in the diagonalization of $A_{n}$. Then, since $A_{n}$ is a companion matrix, a standard choice would be
\begin{equation}
\label{MatPass}
P_{n} = \begin{pmatrix}
1 & 1 & \hdots & 1 \\
\frac{1}{\lambda_{n,\, 1}} & \frac{1}{\lambda_{n,\, 2}} & \hdots & \frac{1}{\lambda_{n,\, p}} \\
\vdots & \vdots & & \vdots \\
\frac{1}{\lambda_{n,\, 1}^{p-1}} & \frac{1}{\lambda_{n,\, 2}^{p-1}} & \hdots & \frac{1}{\lambda_{n,\, p}^{p-1}}
\end{pmatrix}.
\end{equation}
This Vandermonde matrix is invertible if and only if $\lambda_{n,\, i} \neq \lambda_{n,\, j}$ for all $i \neq j$ (see \textit{e.g.} Sec. 0.9.11 of \cite{HornJohnson92}). In that case, $P_{n}^{\, -1}$ is closely related to the Lagrange interpolating polynomials given, for $i \in \{ 1, \hdots, p \}$, by
\begin{equation*}
L_{i}(X) = \frac{\prod_{j\, \neq\, i} (X-\frac{1}{\lambda_{n,\, j}})}{\prod_{j\, \neq\, i} (\frac{1}{\lambda_{n,\, i}}-\frac{1}{\lambda_{n,\, j}})}.
\end{equation*}
Precisely, the $i$--th row of $P_{n}^{\, -1}$ contains the coefficients of $L_{i}(X)$ in the basis $(1, X, \hdots, X^{p-1})$ of $\dR_{p-1}[X]$, \textit{i.e.}
\begin{equation}
\label{MatPassInv}
P_{n}^{\, -1} = \Bigg( \frac{p_{n,\, i,\, j}}{\prod_{j\, \neq\, i} {(\frac{1}{\lambda_{n,\, i}}-\frac{1}{\lambda_{n,\, j}})}} \Bigg)_{\!1\, \leq\, i,j\, \leq\, p}
\end{equation}
where the relation $\prod_{j\, \neq\, i} (X-\frac{1}{\lambda_{n,\, j}}) = p_{n,\, i,\, 1} + p_{n,\, i,\, 2}\, X + \hdots + p_{n,\, i,\, p}\, X^{p-1}$ enables to identify each $p_{n,\, i,\, j}$. Combining \eqref{InegInvVP} and \eqref{InegDiffInvVP}, it follows that, for all $n > n_0$,
\begin{equation*}
\nm P_{n} \nm_1\, \leq\, p\, (1 + \delta + \hdots + \delta^{\, p-1})\, \leq\, \cst.
\end{equation*}
We also have $\nm P_{n}^{\, -1} \nm_1 \leq \cst$ since $\epsilon_1^{\, p-1} < \prod_{j\, \neq\, i} \vert \frac{1}{\lambda_{n,\, i}}-\frac{1}{\lambda_{n,\, j}} \vert < \epsilon_2^{\, p-1}$ and since $p_{n,\, i,\, j}$ is a finite combination of sums and products of $\frac{1}{\lambda_{n,\, 1}}, \hdots, \frac{1}{\lambda_{n,\, p}}$. To sum up, for all $\ell \geq 0$ and $n > n_0$,
\begin{equation*}
A_{n}^{\ell} = P_{n}\, D_{n}^{\, \ell}\, P_{n}^{\, -1} \hsp \text{where} \hsp D_{n} = \diag(\lambda_{n,\, 1}, \hdots, \lambda_{n,\, p}).
\end{equation*}
Consequently,
\begin{eqnarray}
\label{InegNormAn}
\nm A_{n}^{\ell} \nm & = & \nm A_{n}^{\ell} \nm\, \ind_{\{ n\, \leq\, n_0 \}} + \nm P_{n}\, D_{n}^{\, \ell}\, P_{n}^{\, -1} \nm\, \ind_{\{ n\, >\, n_0 \}} \nonumber \\
 & \leq & \nm A_{n}^{\ell} \nm\, \ind_{\{ n\, \leq\, n_0 \}} + \nm P_{n} \nm\, \nm P_{n}^{\, -1} \nm\, \nm D_{n}^{\, \ell} \nm\, \ind_{\{ n\, >\, n_0 \}} \nonumber \\
 & \leq & \nm A_{n}^{\ell} \nm\, \ind_{\{ n\, \leq\, n_0 \}} + \cst\, \rho^{\ell}(A_{n})\, \ind_{\{ n\, >\, n_0 \}}.
\end{eqnarray}
It only remains to sum over $\ell$ and to let $n$ tend to infinity to reach the first result. Similarly,
\begin{equation*}
(\ell+1)\, \nm A_{n}^{\ell} \nm ~ \leq ~ (\ell+1)\, \nm A_{n}^{\ell} \nm\, \ind_{\{ n\, \leq\, n_0 \}} + \cst\, (\ell+1)\, \rho^{\ell}(A_{n})\, \ind_{\{ n\, >\, n_0 \}}
\end{equation*}
so we get the second result by following the same lines.
\end{proof}

\begin{lem}
\label{LemCvgPuissAn}
Under hypotheses \ref{HypCA} and \ref{HypSR}, we have the convergence
\begin{equation*}
\limn A_{n}^{w_{n}} = 0
\end{equation*}
for any rate $(w_{n})$ satisfying $w_{n}\, (1-\rho(A_{n})) \rightarrow +\infty$.
\end{lem}
\begin{proof}
Consider the rank $n_0$ introduced in the proof of Lemma \ref{LemEquivLn}. Then, according to the inequality \eqref{InegNormAn},
\begin{equation}
\label{InegNormA}
\nm A_{n}^{w_{n}} \nm ~\leq~ \nm A_{n}^{w_{n}} \nm\, \ind_{\{ n\, \leq\, n_0 \}} + \cst\, \rho^{\, w_{n}}(A_{n})\, \ind_{\{ n\, >\, n_0 \}}
\end{equation}
where the invertible and uniformly bounded matrices $P_{n}$ and $P_{n}^{\, -1}$ are given in \eqref{MatPass} and \eqref{MatPassInv}, respectively. We also have
\begin{equation}
\label{LimSR}
\limn \rho^{\, w_{n}}(A_{n}) = \limn \de^{\, -w_{n}\, (1-\rho(A_{n}))} = 0
\end{equation}
from the hypothesis on $(w_{n})$. It remains to let $n$ tend to infinity in the above inequality.
\end{proof}

\subsection{Proofs of the main results}
\label{SecProThm}

First of all, it is convenient to express the empirical variance of the process as
\begin{eqnarray*}
\frac{1}{n} \sum_{k=1}^{n} (\Phi_{n,\, k}\, \Phi_{n,\, k}\tr - \Gamma_{n}) & = & \frac{1}{n} \sum_{k=1}^{n} A_{n}\, \Phi_{n,\, k-1}\, \Phi_{n,\, k-1}\tr\, A_{n}\tr + \frac{1}{n} \sum_{k=1}^{n} A_{n}\, \Phi_{n,\, k-1}\, E_{k}\tr \\
 & & \hsp +~ \frac{1}{n} \sum_{k=1}^{n} E_{k}\, \Phi_{n,\, k-1}\tr\, A_{n}\tr + \frac{1}{n} \sum_{k=1}^{n} E_{k}\, E_{k}\tr - \Gamma_{n} \\
 & = & \frac{1}{n} \sum_{k=1}^{n} \Delta_{n,\, k} + \frac{1}{n} \sum_{k=1}^{n} A_{n}\, (\Phi_{n,\, k}\, \Phi_{n,\, k}\tr - \Gamma_{n})\, A_{n}\tr - \frac{T_{n}}{n}
\end{eqnarray*}
where the variance $\Gamma_{n}$ is given in \eqref{Var},
\begin{eqnarray}
\label{Delta}
\Delta_{n} & = & \frac{1}{n} \sum_{k=1}^{n} \Delta_{n,\, k} \\
 & = & \frac{1}{n} \sum_{k=1}^{n} (A_{n}\, \Phi_{n,\, k-1}\, E_{k}\tr + E_{k}\, \Phi_{n,\, k-1}\tr\, A_{n}\tr + E_{k}\, E_{k}\tr + A_{n}\, \Gamma_{n}\, A_{n}\tr - \Gamma_{n}) \nonumber
\end{eqnarray}
and the residual term is
\begin{equation*}
T_{n} = A_{n}\, (\Phi_{n,\, n}\, \Phi_{n,\, n}\tr - \Phi_{n,\, 0}\, \Phi_{n,\, 0}\tr)\, A_{n}\tr.
\end{equation*}
Then, solving this generalized Sylvester equation (Lem. 2.1 of \cite{JiangWei03}) and considering the invertibility of $B_{n}$ in \eqref{MatB} which is proved at the beggining of Section \ref{SecComm}, we reach the decomposition
\begin{equation}
\label{DecompVar}
\vec\bigg( \frac{1}{n} \sum_{k=1}^{n} (\Phi_{n,\, k}\, \Phi_{n,\, k}\tr - \Gamma_{n}) \bigg) = B_{n}^{-1}\, \vec(\Delta_{n}) - \frac{B_{n}^{-1}\, \vec(T_{n})}{n}.
\end{equation}
Let us now reason step by step, \textit{via} some intermediate results.

\subsubsection{Exponential moments of the squared initial value}
\label{SecInit}

We recall that, from the causal form \eqref{CausalForm} of the process,
\begin{equation*}
\Phi_{n,\, 0} = \sum_{\ell=0}^{+\infty} A_{n}^{\ell}\, E_{-\ell}.
\end{equation*}
The following result gives an exponential moment for the correctly renormalized squared initial value.

\begin{lem}
\label{LemExpInit}
Under hypothesis \ref{HypGI},
\begin{equation*}
\dE\bigg[ \exp\Big( \frac{\alpha}{L_{n}^2}\, \nm \Phi_{n,\, 0}\, \Phi_{n,\, 0}\tr \nm \Big) \bigg] < +\infty
\end{equation*}
where $L_{n}$ is given in \eqref{Ln}.
\end{lem}
\begin{proof}
By Cauchy-Schwarz inequality,
\begin{eqnarray*}
\nm \Phi_{n,\, 0}\, \Phi_{n,\, 0}\tr \nm ~ \leq ~ \Vert \Phi_{n,\, 0} \Vert^2 & \leq & \bigg( \sum_{\ell=0}^{+\infty} \Vert A_{n}^{\ell}\, E_{-\ell} \Vert \bigg)^{\! 2} \\
 & \leq & \bigg( \sum_{\ell=0}^{+\infty} \nm A_{n}^{\ell} \nm^{\frac{1}{2}}\, \nm A_{n}^{\ell} \nm^{\frac{1}{2}}\, \Vert E_{-\ell} \Vert \bigg)^{\! 2} ~ \leq ~ L_{n} \sum_{\ell=0}^{+\infty} \nm A_{n}^{\ell} \nm\, \veps_{-\ell}^{\, 2}.
\end{eqnarray*}
Moreover, from Jensen's inequality, for all $\lambda > 0$,
\begin{equation*}
\exp\bigg( \frac{\lambda}{L_{n}} \sum_{\ell=0}^{+\infty} \nm A_{n}^{\ell} \nm\, \veps_{-\ell}^{\, 2} \bigg)\, \leq\, \frac{1}{L_{n}} \sum_{\ell=0}^{+\infty} \nm A_{n}^{\ell} \nm\, \de^{\lambda\, \veps_{-\ell}^{\, 2}}
\end{equation*}
using $\frac{\nm A_{n}^{0} \nm}{L_{n}} + \frac{\nm A_{n}^{1} \nm}{L_{n}} + \hdots = 1$. Taking the expectation and choosing $\lambda = \alpha$ given in \ref{HypGI}, we deduce that
\begin{eqnarray}
\label{ExpMomInit}
\dE\bigg[ \exp\Big( \frac{\alpha}{L_{n}^2}\, \nm \Phi_{n,\, 0}\, \Phi_{n,\, 0}\tr \nm \Big) \bigg] & \leq & \frac{1}{L_{n}} \sum_{\ell=0}^{+\infty} \nm A_{n}^{\ell} \nm\, \dE\big[ \de^{\alpha\, \veps_{-\ell}^{\, 2}} \big] \nonumber \\
 & = & \dE\big[ \de^{\alpha\, \veps_1^{\, 2}} \big] < +\infty.
\end{eqnarray}
\end{proof}

\subsubsection{Exponential convergence of the residual term}
\label{SecRes}

The residual term in the decomposition \eqref{DecompVar} is given by
\begin{equation}
\label{Res}
R_{n} = \frac{B_{n}^{-1}\, \vec( A_{n}\, (\Phi_{n,\, n}\, \Phi_{n,\, n}\tr - \Phi_{n,\, 0}\, \Phi_{n,\, 0}\tr)\, A_{n}\tr )}{n}.
\end{equation}
Our next objective is to prove the exponential negligibility of this residual.

\begin{lem}
\label{LemExpRes}
Under hypotheses \ref{HypGI}--\ref{HypMD}, for all $r > 0$,
\begin{equation*}
\limn \frac{1}{b_{n}^{\, 2}} \ln \dP\bigg( \frac{\sqrt{n}\, (1 - \rho(A_{n}))^{\frac{3}{2}}}{b_{n}}\, \Vert R_{n} \Vert \geq r \bigg) = -\infty.
\end{equation*}
\end{lem}
\begin{proof}
First, note that
\begin{eqnarray*}
\Vert R_{n} \Vert & \leq & \frac{\nm B_{n}^{-1} \nm\, \Vert \vec( A_{n}\, (\Phi_{n,\, n}\, \Phi_{n,\, n}\tr - \Phi_{n,\, 0}\, \Phi_{n,\, 0}\tr)\, A_{n}\tr ) \Vert}{n} \\
 & \leq & \frac{\cst\, \nm B_{n}^{-1} \nm\, \nm A_{n} \nm^2\, \nm \Phi_{n,\, n}\, \Phi_{n,\, n}\tr - \Phi_{n,\, 0}\, \Phi_{n,\, 0}\tr \nm}{n} \\
 & \leq & \frac{\cst\, \nm B_{n}^{-1} \nm\, \nm A_{n} \nm^2\, (\nm \Phi_{n,\, n}\, \Phi_{n,\, n}\tr \nm + \nm \Phi_{n,\, 0}\, \Phi_{n,\, 0}\tr \nm)}{n}.
\end{eqnarray*}
Thus,
\begin{eqnarray*}
\dP\bigg( \frac{\sqrt{n}\, (1 - \rho(A_{n}))^{\frac{3}{2}}}{b_{n}}\, \Vert R_{n} \Vert \geq r \bigg) & = & \dP\bigg( \Vert R_{n} \Vert \geq \frac{r\, b_{n}\, (1 - \rho(A_{n}))^{-\frac{3}{2}}}{\sqrt{n}} \bigg) \\
 & \leq & 2\, \dP\bigg( \nm \Phi_{n,\, 0}\, \Phi_{n,\, 0}\tr \nm \geq \frac{r\, b_{n} \sqrt{n}\, (1 - \rho(A_{n}))^{-\frac{3}{2}}}{2\, \cst\, \nm A_{n} \nm^2\, \nm B_{n}^{-1} \nm} \bigg) \\
 & \leq & 2\, \dE\big[ \de^{\alpha\, \veps_1^{\, 2}} \big] \exp\bigg(\! -\frac{r\, \alpha\, b_{n} \sqrt{n}\, (1 - \rho(A_{n}))^{-\frac{3}{2}}}{2\, \cst\, \nm A_{n} \nm^2\, \nm B_{n}^{-1} \nm\, L_{n}^2} \bigg)
\end{eqnarray*}
where $L_{n}$ is given in \eqref{Ln}, using Markov's inequality, the reasoning in the proof of Lemma \ref{LemExpInit} and the fact that, from the strict stationarity of the process, $\Phi_{n,\, 0}\, \Phi_{n,\, 0}\tr$ and $\Phi_{n,\, n}\, \Phi_{n,\, n}\tr$ share the same distribution. Hence, for a sufficiently large $n$,
\begin{eqnarray*}
\frac{1}{b_{n}^{\, 2}} \ln \dP\bigg( \frac{\sqrt{n}\, (1 - \rho(A_{n}))^{\frac{3}{2}}}{b_{n}}\, \Vert R_{n} \Vert \geq r \bigg) & \leq & \frac{\ln 2 + \ln \dE\big[ \de^{\alpha\, \veps_1^{\, 2}} \big]}{b_{n}^{\, 2}} - \frac{r\, \alpha\, \sqrt{n}\, (1 - \rho(A_{n}))^{-\frac{3}{2}}}{2\, \cst\, b_{n}\, \nm A_{n} \nm^2\, \nm B_{n}^{-1} \nm\, L_{n}^2} \\
 & \leq & \frac{\ln 2 + \ln \dE\big[ \de^{\alpha\, \veps_1^{\, 2}} \big]}{b_{n}^{\, 2}} - \cst\, \frac{\sqrt{n}\, (1 - \rho(A_{n}))^{\frac{3}{2}}}{b_{n}}
\end{eqnarray*}
since $\nm B_{n}^{-1} \nm^{\frac{1}{2}} \sim \sqrt{h}\, (1 - \rho(A_{n}))^{-\frac{1}{2}}$ from \ref{HypRC}, $L_{n}^2 = O((1 - \rho(A_{n}))^{-2})$ from Lemma \ref{LemEquivLn} and since, from \ref{HypCA}, $\nm A_{n} \nm$ converges. Finally, letting $n$ tend to infinity, \ref{HypGI} and \ref{HypMD} conclude the proof.
\end{proof}

\subsubsection{The truncated sequence}
\label{SecTrunc}

In what follows, we define the rate
\begin{equation}
\label{RateMn}
m_{n} = \left\lfloor \bigg( \frac{1}{1 - \rho(A_{n})} \bigg)^{\! \frac{3 + 3 \eta}{3 + 2 \eta}} \right\rfloor
\end{equation}
and we note from \ref{HypSR}--\ref{HypMD} that
\begin{equation}
\label{PropMn}
\limn m_{n} (1 - \rho(A_{n})) = +\infty \hsp \text{and} \hsp \limn \frac{b_{n}\, \nm B_{n}^{-1} \nm^{\frac{1}{2}}\, m_{n}^{1+\frac{2 \eta}{3}}}{\sqrt{n}} = 0.
\end{equation}
Following the idea of \cite{MiaoWangYang15}, we are going to use $m_{n}$ as a truncation parameter. Consider
\begin{equation}
\label{PsiTrunc}
\Psi_{n,\, k} = \sum_{\ell=0}^{m_{n}-2} A_{n}^{\ell}\, E_{k-\ell}
\end{equation}
as an approximation of $\Phi_{n,\, k}$ in its causal form \eqref{CausalForm}. We also define the truncated version of the summands $\Delta_{n,\, k}$ in \eqref{Delta} as
\begin{equation}
\label{ZetaTrunc}
\zeta_{n,\, k} = A_{n}\, \Psi_{n,\, k-1}\, E_{k}\tr + E_{k}\, \Psi_{n,\, k-1}\tr\, A_{n}\tr + E_{k}\, E_{k}\tr + A_{n}\, \Gamma_{n}\, A_{n}\tr - \Gamma_{n}.
\end{equation}
The process $(B_{n}^{-1}\, \vec(\zeta_{n,\, k}))_{k}$ is strictly stationary and $m_{n}$--dependent, according to Def. 6.4.3 of \cite{BrockwellDavis91}. Let us study some properties of this process.

\begin{lem}
\label{LemMomExpTrunc}
Under hypotheses \ref{HypGI}--\ref{HypRC}, we can find a constant $c_{\alpha} > 0$ such that, for a sufficiently large $n$,
\begin{equation*}
\dE\bigg[ \exp\Big( c_{\alpha}\, \nm B_{n}^{-1} \nm^{-1}\, \sum_{\ell=0}^{w_{n}} \nm A_{n}^{\ell}\, E_{-\ell}\, E_1\tr \nm \Big) \bigg]\, \leq\, \dE\big[ \de^{\alpha\, \veps_1^{\, 2}} \big]
\end{equation*}
for any rate $(w_{n})$ satisfying $w_{n}\, (1-\rho(A_{n})) \rightarrow +\infty$.
\end{lem}
\begin{proof}
By H\"older's inequality,
\begin{equation*}
\dE\bigg[ \exp\Big( c_{\alpha}\, \nm B_{n}^{-1} \nm^{-1}\, \sum_{\ell=0}^{w_{n}} \nm A_{n}^{\ell}\, E_{-\ell}\, E_1\tr \nm \Big) \bigg]\, \leq\, \dE\bigg[ \exp\Big( c_{\alpha}\, \nm B_{n}^{-1} \nm^{-1}\, \sum_{\ell=0}^{w_{n}} \nm A_{n}^{\ell} \nm\, \veps_1^{\, 2} \Big) \bigg].
\end{equation*}
Moreover, for the rank $n_0$ and the uniformly bounded matrices $P_{n}$ and $P_{n}^{\, -1}$ introduced in the proof of Lemma \ref{LemEquivLn},
\begin{eqnarray*}
\sum_{\ell=0}^{w_{n}} \nm A_{n}^{\ell} \nm & = & \sum_{\ell=0}^{n_0} \nm A_{n}^{\ell} \nm + \sum_{\ell=n_0+1}^{w_{n}} \nm A_{n}^{\ell} \nm \\
 & = & \sum_{\ell=0}^{n_0} \nm A_{n}^{\ell} \nm + \sum_{\ell=n_0+1}^{w_{n}} \nm P_{n}\, D_{n}^{\, \ell}\, P_{n}^{\, -1} \nm ~ \leq ~ \cst\, \bigg( 1 + \frac{1 - \rho^{\, w_{n}}(A_{n}) }{1 - \rho(A_{n})} \bigg)
\end{eqnarray*}
as soon as $w_{n} > n_0$. Thus,
\begin{equation*}
\nm B_{n}^{-1} \nm^{-1}\, \sum_{\ell=0}^{w_{n}} \nm A_{n}^{\ell} \nm\, \leq\, \cst\, \bigg( \nm B_{n}^{-1} \nm^{-1} + \frac{1 - \rho^{\, w_{n}}(A_{n})}{\nm B_{n}^{-1} \nm\, (1 - \rho(A_{n}))} \bigg).
\end{equation*}
Finally, \ref{HypRC}, \eqref{LimBL} and \eqref{LimSR} lead, for large values of $n$, to
\begin{equation*}
\nm B_{n}^{-1} \nm^{-1}\, \sum_{\ell=0}^{w_{n}} \nm A_{n}^{\ell} \nm\, \leq\, \cst.
\end{equation*}
It remains to choose $c_{\alpha} = \frac{\alpha}{\cst}$.
\end{proof}

\begin{lem}
\label{LemVarTrunc}
Under hypotheses \ref{HypCA}--\ref{HypRC}, for all $n \geq 1$ and $k \in \{ 1, \hdots, n \}$,
\begin{equation*}
\dE[\vec(\zeta_{n,\, k})] = 0 \hsp \text{and} \hsp \cov(\vec(\zeta_{n,\, k}), \vec(\zeta_{n,\, j})) = \left\{
\begin{array}{ll}
0 & \mbox{for } k \neq j \\
\Upsilon_{n} & \mbox{for } k = j
\end{array}
\right. 
\end{equation*}
where the $p^2 \times p^2$ covariance $\Upsilon_{n}$ can be explicitely built in terms of $\sigma^2$, $A_{n}$ and $B_{n}$. In addition,
\begin{equation*}
\limn \frac{B_{n}^{-1}\, \Upsilon_{n}\, (B_{n}^{-1})\tr}{\nm B_{n}^{-1} \nm^{\, 3}} = \Upsilon
\end{equation*}
where the non-zero limiting matrix $\Upsilon$ is given in \eqref{LimVarRenorm}.
\end{lem}
\begin{proof}
We will use in what follows $K_{p}$ and $U_{p}$ defined in \eqref{MatKU}. Let $\cF_{k} = \sigma(\veps_{\ell},\, \ell \leq k)$ be the $\sigma$--algebra of the events occurring up to time $k$. Then, it is easy to see that
\begin{eqnarray*}
\dE[\vec(\zeta_{n,\, k})] & = & \dE[\, \dE[\vec(\zeta_{n,\, k})\, \vert\, \cF_{k-1}]\, ] \\
 & = & \sigma^2\, \vec(K_{p}) + \vec(A_{n}\, \Gamma_{n}\, A_{n}\tr - \Gamma_{n}) ~ = ~ 0
\end{eqnarray*}
in virtue of \eqref{RelGammaK}. For $k > j$, by direct calculation,
\begin{eqnarray*}
\dE[\vec(\zeta_{n,\, k})\, \vec\tr(\zeta_{n,\, j})] & = & \dE[\, \dE[\vec(\zeta_{n,\, k})\, \vec\tr(\zeta_{n,\, j})\, \vert\, \cF_{k-1}]\, ] \\
 & = & \dE[\, (\dE[\vec(A_{n}\, \Psi_{n,\, k-1}\, E_{k}\tr) + \vec(E_{k}\, \Psi_{n,\, k-1}\tr\, A_{n}\tr) \vert\, \cF_{k-1}] \\
 & & \hsp +~ \sigma^2\, \vec(K_{p}) + \vec(A_{n}\, \Gamma_{n}\, A_{n}\tr - \Gamma_{n}))\, \vec\tr(\zeta_{n,\, j}) ] ~ = ~ 0
\end{eqnarray*}
and the same is true for $j > k$ since $(\dE[\vec(\zeta_{n,\, k})\, \vec\tr(\zeta_{n,\, j})])\tr = \dE[\vec(\zeta_{n,\, j})\, \vec\tr(\zeta_{n,\, k})] = 0$. Now for $k = j$, a tedious but straightforward calculation leads to
\begin{eqnarray}
\label{VarVec}
\dE[\vec(\zeta_{n,\, k})\, \vec\tr(\zeta_{n,\, k})] & = & \sigma^2\, K_{p} \otimes (A_{n}\, \dE[\Psi_{n,\, k-1} \Psi_{n,\, k-1}\tr]\, A_{n}\tr) \nonumber \\
 & & \hsp +~ \sigma^2\, U_{p} \otimes (A_{n}\, \dE[\Psi_{n,\, k-1} \Psi_{n,\, k-1}\tr]\, A_{n}\tr) \otimes U_{p}\tr \nonumber \\
 & & \hsp +~ \sigma^2\, U_{p}\tr \otimes (A_{n}\, \dE[\Psi_{n,\, k-1} \Psi_{n,\, k-1}\tr]\, A_{n}\tr) \otimes U_{p} \nonumber \\
 & & \hsp +~ \sigma^2\, (A_{n}\, \dE[\Psi_{n,\, k-1} \Psi_{n,\, k-1}\tr]\, A_{n}\tr) \otimes K_{p} \nonumber \\
 & & \hsp +~ (\tau^4 - \sigma^4)\, \vec(K_{p})\, \vec\tr(K_{p}) ~ = ~ \Upsilon_{n}.
\end{eqnarray}
To give an explicit expression of $\Upsilon_{n}$, it suffices to observe that the truncated expression \eqref{PsiTrunc} has a variance given by
\begin{equation*}
\Gamma_{n,\, m_{n}} = \dE[\Psi_{n,\, k-1} \Psi_{n,\, k-1}\tr] = \sigma^2 \sum_{\ell = 0}^{m_{n}-2} A_{n}^{\ell}\, K_{p}\, (A_{n}\tr)^{\ell} 
\end{equation*}
so that
\begin{eqnarray*}
\vec( \Gamma_{n,\, m_{n}} ) & = & \sigma^2 \sum_{\ell = 0}^{m_{n}-2} (A_{n} \otimes A_{n})^{\ell}\, \vec(K_{p}) \\
 & = & \sigma^2\, B_{n}^{-1}\, (I_{p^2} - (A_{n} \otimes A_{n})^{m_{n}-1})\, \vec(K_{p}).
\end{eqnarray*}
Let us now look at the asymptotic behavior of $\Upsilon_{n}$ correctly renormalized. First, we have the convergence
\begin{equation*}
\limn (A_{n} \otimes A_{n})^{m_{n}-1} = 0
\end{equation*}
coming from the identity $(A_{n} \otimes A_{n})^{m_{n}-1} = A_{n}^{m_{n}-1} \otimes A_{n}^{m_{n}-1}$ and Lemma \ref{LemCvgPuissAn}. Together with \ref{HypRC}, this implies
\begin{equation*}
\limn \frac{\vec( \Gamma_{n,\, m_{n}} )}{\nm B_{n}^{-1} \nm} = \sigma^2\, H\, \vec(K_{p}).
\end{equation*}
In the end of the proof, we call $\vec^{-1}$ the vectorization inverse operator (namely, in our context, the reconstruction of a $p \times p$ matrix from its vectorization of size $p^2$). Then,
\begin{equation}
\label{LimGam}
\limn \frac{\Gamma_{n,\, m_{n}}}{\nm B_{n}^{-1} \nm} = \sigma^2\, \vec^{-1}(H\, \vec(K_{p})) = \Gamma.
\end{equation}
Combining \eqref{VarVec} with \eqref{LimGam} and \ref{HypRC}, we have
\begin{equation}
\label{LimVarRenorm}
\Upsilon = \sigma^2\, H\, (K_{p} \otimes \Gamma^{A} + U_{p} \otimes \Gamma^{A} \otimes U_{p}\tr  + U_{p}\tr \otimes \Gamma^{A} \otimes U_{p} + \Gamma^{A} \otimes K_{p})\, H\tr
\end{equation}
where $\Gamma^{A} = A\, \Gamma\, A\tr$.
\end{proof}

\begin{rem}
As a by-product, we also obtain, following the same lines,
\begin{equation*}
\limn \frac{\Gamma_{n}}{\nm B_{n}^{-1} \nm} = \Gamma
\end{equation*}
where $\Gamma_{n}$ is given in \eqref{Var}, which proves \eqref{LimGamGen}. The variance $\Gamma_{n,\, m_{n}}$ defined above may be seen as the truncated version of $\Gamma_{n}$.
\end{rem}

\subsubsection{The remainder of the truncation}
\label{SecResTrunc}

We denote by
\begin{equation}
\label{ResTrunc}
\Lambda_{n} = \frac{1}{n} \sum_{k=1}^{n} (A_{n}\, (\Phi_{n,\, k-1} - \Psi_{n,\, k-1})\, E_{k}\tr + E_{k}\, (\Phi_{n,\, k-1} - \Psi_{n,\, k-1})\tr\, A_{n}\tr)
\end{equation}
the remainder of the truncation of $\Delta_{n}$ in \eqref{Delta} made \textit{via} \eqref{ZetaTrunc}. Our last preliminary objective is to establish the following lemma.
\begin{lem}
\label{LemExpResTrunc}
Under hypotheses \ref{HypGI}--\ref{HypMD}, for all $r > 0$,
\begin{equation*}
\limn \frac{1}{b_{n}^{\, 2}} \ln \dP\bigg( \frac{\sqrt{n}\, (1 - \rho(A_{n}))^{\frac{3}{2}}}{b_{n}}\, \Vert B_{n}^{-1}\, \vec(\Lambda_{n}) \Vert \geq r \bigg) = -\infty.
\end{equation*}
\end{lem}
\begin{proof}
Clearly, both terms in the definition of \eqref{ResTrunc} are similar and we will only work on the first one. From the causal expression \eqref{CausalForm} and the truncation \eqref{PsiTrunc}, we note that
\begin{eqnarray*}
\sum_{k=1}^{n} A_{n}\, (\Phi_{n,\, k-1} - \Psi_{n,\, k-1})\, E_{k}\tr & = & \sum_{k=1}^{n} \sum_{\ell=m_{n}-1}^{+\infty} A_{n}^{\ell+1}\, E_{k-1-\ell}\, E_{k}\tr \\
 & = & A_{n}^{m_{n}} \sum_{\ell=0}^{+\infty} A_{n}^{\ell}\, \sum_{k=1}^{n} E_{k-\ell-m_{n}}\, E_{k}\tr.
\end{eqnarray*}
Thus, with $M_{n}$ given in \eqref{Mn} and applying Lem. 17 of \cite{MasMenneteau03} under \ref{HypGI},
\begin{flalign}
\label{InegResExpTrunc}
& \dP\bigg( \frac{1}{n}\, \nmg \sum_{k=1}^{n} A_{n}\, (\Phi_{n,\, k-1} - \Psi_{n,\, k-1})\, E_{k}\tr \nmg \geq r\, \frac{b_{n}}{\sqrt{n}}\, \nm B_{n}^{-1} \nm^{\frac{1}{2}} \bigg) \nonumber \\
& \hsp \hsp \hsp \leq ~ \dP\bigg( \sum_{\ell=0}^{+\infty} \nm A_{n}^{\ell} \nm\, \Big\vert \sum_{k=1}^{n} \veps_{k-\ell-m_{n}}\, \veps_{k} \Big\vert \geq \sum_{\ell=0}^{+\infty} (\ell+1)\, \nm A_{n}^{\ell} \nm\, \frac{r\, b_{n}\, \sqrt{n}\, \nm B_{n}^{-1} \nm^{\frac{1}{2}}}{M_{n}\, \nm A_{n}^{m_{n}} \nm} \bigg) \nonumber \\
& \hsp \hsp \hsp \leq ~ \sum_{\ell=0}^{+\infty} \dP\bigg( \max_{1\, \leq\, j\, \leq\, n} \Big\vert \sum_{k=1}^{j} \veps_{k-\ell-m_{n}}\, \veps_{k} \Big\vert \geq \frac{r\, (\ell+1)\, b_{n}\, \sqrt{n}\, \nm B_{n}^{-1} \nm^{\frac{1}{2}}}{M_{n}\, \nm A_{n}^{m_{n}} \nm} \bigg) \nonumber \\
& \hsp \hsp \hsp \leq ~ \cst\, \sum_{\ell=0}^{+\infty} \exp\Big(\!-\frac{r^2\, b_{n}^{\, 2}\, n\, t_{n,\, \ell}^{\, 2}}{\alpha_0\, n + \beta_0\, r\, t_{n,\, \ell}\, b_{n}\, \sqrt{n}} \Big)
\end{flalign}
for some $\alpha_0 > 0$ and $\beta_0 > 0$, where
\begin{equation*}
t_{n,\, \ell} = \frac{(\ell+1)\, \nm B_{n}^{-1} \nm^{\frac{1}{2}}}{M_{n}\, \nm A_{n}^{m_{n}} \nm}.
\end{equation*}
Our choice of $m_{n}$ in \eqref{Mn}, the properties of Lemma \ref{LemEquivLn}, \eqref{InegNormA} and our hypotheses on the rates of convergence lead, for $n$ large enough, to
\begin{equation*}
\nm B_{n}^{-1} \nm^{-\frac{1}{2}}\, M_{n}\, \nm A_{n}^{m_{n}} \nm ~ \leq ~ \cst\, (1 - \rho(A_{n}))^{-\frac{3}{2}}\, \rho^{\, m_{n}}(A_{n}) ~ \longrightarrow ~ 0
\end{equation*}
and obviously $t_{n,\, \ell} \rightarrow +\infty$. Hence, like in formula (3.11) of \cite{MiaoWangYang15}, there are some constants $\alpha_0^{\, \prime} > 0$ and $\beta_0^{\, \prime} > 0$ such that, for all $\ell \geq 0$ and large values of $n$,
\begin{eqnarray*}
\frac{r^2\, n\, b_{n}^{\, 2}\, t_{n,\, \ell}^{\, 2}}{\alpha_0\, n + \beta_0\, r\, t_{n,\, \ell}\, b_{n}\, \sqrt{n}} & = & \frac{r^2\, (\ell+1)\, b_{n}^{\, 2}\, t_{n,\, \ell}}{\alpha_0\, \nm B_{n}^{-1} \nm^{-\frac{1}{2}}\, M_{n}\, \nm A_{n}^{m_{n}} \nm + r\, \beta_0\, (\ell + 1)\, \frac{b_{n}}{\sqrt{n}}} \\
 & \geq & b_{n}^{\, 2}\, t_{n,\, \ell}\, \frac{r^2}{\alpha_0^{\, \prime} + r\, \beta_0^{\, \prime}}.
\end{eqnarray*}
Going back to \eqref{InegResExpTrunc},
\begin{eqnarray*}
\sum_{\ell=0}^{+\infty} \exp\Big(\!-\frac{r^2\, b_{n}^{\, 2}\, n\, t_{n,\, \ell}^{\, 2}}{\alpha_0\, n + \beta_0\, r\, t_{n,\, \ell}\, b_{n}\, \sqrt{n}} \Big) & \leq & \sum_{\ell=0}^{+\infty} \exp\Big(\!- b_{n}^{\, 2}\, t_{n,\, \ell}\, \frac{r^2}{\alpha_0^{\, \prime} + r\, \beta_0^{\, \prime}} \Big) \\
 & = & \frac{\de^{-V_{n}}}{1 - \de^{-V_{n}}}
\end{eqnarray*}
where, for convenience, we note
\begin{equation*}
V_{n} = \frac{r^2\, b_{n}^{\, 2}\, \nm B_{n}^{-1} \nm^{\frac{1}{2}}}{M_{n}\, \nm A_{n}^{m_{n}} \nm\, (\alpha_0^{\, \prime} + r\, \beta_0^{\, \prime})} ~ \longrightarrow ~ +\infty.
\end{equation*}
To sum up,
\begin{flalign*}
& \frac{1}{b_{n}^{\, 2}} \ln \dP\bigg( \nmg \sum_{k=1}^{n} A_{n}\, (\Phi_{n,\, k-1} - \Psi_{n,\, k-1})\, E_{k}\tr \nmg \geq r\, b_{n}\, \sqrt{n}\, \nm B_{n}^{-1} \nm^{\frac{1}{2}} \bigg) \\
& \hsp \hsp \hsp \leq ~ \frac{\cst - \ln(1 - \de^{-V_{n}})}{b_{n}^{\, 2}} - \frac{V_{n}}{b_{n}^{\, 2}} \\
& \hsp \hsp \hsp \leq ~ \frac{\cst - \ln(1 - \de^{-V_{n}})}{b_{n}^{\, 2}} - \frac{r^2\, \nm B_{n}^{-1} \nm^{\frac{1}{2}}}{M_{n}\, \nm A_{n}^{m_{n}} \nm\, (\alpha_0^{\, \prime} + r\, \beta_0^{\, \prime})} ~ \longrightarrow ~ -\infty.
\end{flalign*}
This is clearly sufficient to finish the proof since, from \ref{HypRC},
\begin{eqnarray*}
\frac{\sqrt{n}\, (1 - \rho(A_{n}))^{\frac{3}{2}}}{b_{n}}\, \Vert B_{n}^{-1}\, \vec(\Lambda_{n}) \Vert & \leq & \cst\, \frac{\sqrt{n}}{b_{n}}\, \nm B_{n}^{-1} \nm^{-\frac{1}{2}}\, \Vert \vec(\Lambda_{n}) \Vert \\
 & \leq & \cst\, \frac{\sqrt{n}}{b_{n}}\, \nm B_{n}^{-1} \nm^{-\frac{1}{2}}\, \nm \Lambda_{n} \nm
\end{eqnarray*}
for $n$ large enough.
\end{proof}
\noindent We are now ready to prove Theorem \ref{ThmMDPCov} and Corollary \ref{CorMDPOLS}.

\subsubsection{Proof of Theorem \ref{ThmMDPCov}}
\label{SecProMDPCov}

All the technical results of the previous sections are now going to be concretely used. Consider the sequence
\begin{equation}
\label{XiTrunc}
\xi_{n,\, k} = \frac{B_{n}^{-1}\, \vec(\zeta_{n,\, k})}{\nm B_{n}^{-1} \nm^{\frac{3}{2}}}
\end{equation}
where $\zeta_{n,\, k}$ is given in \eqref{ZetaTrunc}. The process $(\xi_{n,\, k})_{k}$ is also strictly stationary and $m_{n}$--dependent. Like in \cite{MiaoYang08} or \cite[{suppl. mat.}]{MiaoWangYang15}, let us extract an independent sequence from this process. For $j \in \{ 1, \hdots, j_{n} \}$, define
\begin{equation*}
\xi^{\, \prime}_{n,\, j} = \xi_{n,\, (j-1) m_{n} + 1} + \hdots + \xi_{n,\, j m_{n}}
\end{equation*}
where $j_{n} = \lfloor \frac{n}{m_{n}} \rfloor$ and where $(m_{n})$ and its properties are given in \eqref{RateMn}. Then, $(\xi^{\, \prime}_{n,\, j})_{j}$ is strictly stationary and $1$--dependent. Next, for $t \in \{ 1, \hdots, t_{n} \}$, define
\begin{equation*}
\xi^{\, \prime \prime}_{n,\, t} = \xi^{\, \prime}_{n,\, (t-1) u_{n} + 1} + \hdots + \xi^{\, \prime}_{n,\, t u_{n}-1}
\end{equation*}
where $t_{n} = \lfloor \frac{j_{n}}{u_{n}} \rfloor$ and $(u_{n})$ is another rate satisfying
\begin{equation}
\label{PropUn}
\limn u_{n} = +\infty \hsp \text{and} \hsp \limn \frac{b_{n}\, \nm B_{n}^{-1} \nm^{\frac{1}{2}}\, (m_{n}\, u_{n})^{1+\frac{2 \eta}{3}}}{\sqrt{n}} = 0.
\end{equation}
To be convinced that such a rate exists, one can use \eqref{PropMn} and the fact that $\vert \ln f_{n} \vert \rightarrow +\infty$ and $f_{n}\, \vert \ln f_{n} \vert^{a} \rightarrow 0$ when $f_{n} \rightarrow 0$. The process $(\xi^{\, \prime \prime}_{n,\, t})_{t}$ is now i.i.d. and the rates satisfy
\begin{equation}
\label{PropTUM}
\limn \frac{t_{n}\, u_{n}\, m_{n}}{n} = 1.
\end{equation}
The reasoning of \cite[{suppl. mat.}]{MiaoWangYang15} does not suit us, so we need to reformulate the establishment of the MDP. First, by a Taylor-Lagrange expansion,
\begin{equation}
\label{TaylorExp}
\exp\Big( \Big\langle \lambda,\, \frac{b_{n}}{\sqrt{n}}\, \xi^{\, \prime \prime}_{n,\, 1} \Big\rangle \Big) = 1 + \frac{b_{n}}{\sqrt{n}}\, \langle \lambda,\,  \xi^{\, \prime \prime}_{n,\, 1} \rangle + \frac{b_{n}^{\, 2}}{2\, n}\, \langle \lambda,\,  \xi^{\, \prime \prime}_{n,\, 1} \rangle^2 + \frac{b_{n}^3}{6\, n^\frac{3}{2}}\, \langle \lambda,\, \xi^{\, \prime \prime}_{n,\, 1} \rangle^3\, \de^{\, \nu_{n}}
\end{equation}
in which the remainder term satisfies, for any $\alpha > 0$,
\begin{eqnarray*}
\de^{\, \alpha\, \nu_{n}} ~ < ~ \exp\Big( \frac{\alpha\, b_{n}}{\sqrt{n}}\, \vert \langle \lambda,\,  \xi^{\, \prime \prime}_{n,\, 1} \rangle \vert \Big) & \leq & \exp\Big( \frac{\alpha\, b_{n}}{\sqrt{n}}\, \Vert \lambda \Vert\, \sum_{\ell=1}^{m_{n} u_{n}} \Vert \xi_{n,\, \ell} \Vert \Big) \\
 & \leq & \exp\Big(\cst\, \frac{b_{n}}{\sqrt{n}}\, \nm B_{n}^{-1} \nm^{-\frac{1}{2}} \sum_{\ell=1}^{m_{n} u_{n}} \nm \zeta_{n,\, \ell} \nm \Big).
\end{eqnarray*}
Now, the random variables $\nm \zeta_{n,\, \ell} \nm$ sharing the same distribution for all $\ell \geq 0$, it follows from H\"older's inequality that,
\begin{eqnarray}
\label{TaylorExpRem}
\dE\big[ \de^{\, \alpha\, \nu_{n}} \big] & < & \dE \Big[ \exp\Big(\cst\, \frac{b_{n}\, m_{n}\, u_{n}}{\sqrt{n}}\, \nm B_{n}^{-1} \nm^{-\frac{1}{2}}\, \nm \zeta_{n,\, 1} \nm \Big) \Big] \nonumber \\
 & = & \dE \Big[ \exp\Big(\cst\, \frac{b_{n}\, \nm B_{n}^{-1} \nm^{\frac{1}{2}}\, m_{n}\, u_{n}}{\sqrt{n}}\, \nm B_{n}^{-1} \nm^{-1}\, \nm \zeta_{n,\, 1} \nm \Big) \Big] ~ < ~ +\infty
\end{eqnarray}
for $n$ large enough, using Lemma \ref{LemMomExpTrunc} with $m_{n}\, (1-\rho(A_{n})) \rightarrow +\infty$ stemming from \eqref{PropMn}, the convergence of $\nm A_{n} \nm$, \ref{HypGI} and treating all the terms of \eqref{ZetaTrunc} similarly. Taking the expectation in \eqref{TaylorExp} and exploiting the independence of the zero-mean process $(\xi^{\, \prime \prime}_{n,\, t})_{t}$, we obtain the decomposition
\begin{eqnarray}
\label{TaylorCGF}
\frac{1}{b_{n}^{\, 2}} \ln \dE\bigg[ \exp\Big( \Big\langle \lambda,\, \frac{b_{n}}{\sqrt{n}}\, \sum_{\ell=1}^{n} \xi_{n,\, \ell} \Big\rangle \Big) \bigg] & \sim & \frac{t_{n}}{b_{n}^{\, 2}}\, \ln \dE\bigg[ \exp\Big( \Big\langle \lambda,\, \frac{b_{n}}{\sqrt{n}}\, \xi^{\, \prime \prime}_{n,\, 1} \Big\rangle \Big) \bigg] \nonumber \\
 & = & \frac{t_{n}}{2\, n}\, \dE\big[ \langle \lambda,\,  \xi^{\, \prime \prime}_{n,\, 1} \rangle^2 \big] + O\bigg( \frac{t_{n}\, b_{n}}{6\, n^\frac{3}{2}}\, \big\vert \dE\big[ \langle \lambda,\, \xi^{\, \prime \prime}_{n,\, 1} \rangle^3\, \de^{\, \nu_{n}} \big] \big\vert \bigg)
\end{eqnarray}
for we can see, as it is done in \cite{MiaoYang08}, that the residual term
\begin{equation*}
\tau_{n} = \sum_{\ell=1}^{n} \xi_{n,\, \ell} - \sum_{\ell=1}^{t_{n}} \xi^{\, \prime \prime}_{n,\, \ell}
\end{equation*}
plays a negligible role in comparison to the main one. To eliminate the third-order term, we first look at the fourth-order moment of $\langle \lambda,\,  \xi^{\, \prime \prime}_{n,\, 1} \rangle$, that is
\begin{eqnarray*}
\dE\big[ \langle \lambda,\,  \xi^{\, \prime \prime}_{n,\, 1} \rangle^4 \big] & \leq & \frac{\cst\, \Vert \lambda \Vert^4}{\nm B_{n}^{-1} \nm^{2}}\, \dE\bigg[ \nmg \sum_{\ell=1}^{m_{n} u_{n}} \zeta_{n,\, \ell}\, \nmg^{4} \bigg].
\end{eqnarray*}
A long but standard calculation shows that
\begin{eqnarray*}
\dE\bigg[ \nmg A_{n} \sum_{\ell=1}^{m_{n} u_{n}} \Psi_{n,\, \ell-1}\, E_{\ell}\tr\, \nmg^{4} \bigg] & \leq & \cst\, \dE\bigg[ \Big\Vert \sum_{\ell=1}^{m_{n} u_{n}} \Psi_{n,\, \ell-1}\, \veps_{\ell}\, \Big\Vert^{4} \bigg] \\
 & = & O\big( (m_{n}\, u_{n}\, \nm B_{n}^{-1} \nm)^2 \big)
\end{eqnarray*}
as $n$ tends to infinity. This result is reached using the strict stationarity of the process, the explicit expression of $X_{n,\, 0}^4$ in terms of $A_{n}^{\ell}$, the inequality \eqref{InegNormA} and, finally, using \ref{HypRC} giving the equivalence between $(1-\rho(A_{n}))^{-2}$ and $\cst\, \nm B_{n}^{-1} \nm^2$. So,
\begin{equation*}
\dE\big[ \langle \lambda,\,  \xi^{\, \prime \prime}_{n,\, 1} \rangle^4 \big] = O(m_{n}^2\, u_{n}^2).
\end{equation*}
By Lyapunov's inequality,
\begin{equation*}
\dE\big[ \vert \langle \lambda,\,  \xi^{\, \prime \prime}_{n,\, 1} \rangle \vert^{3+\delta} \big]\, \leq\, \big( \dE\big[ \langle \lambda,\,  \xi^{\, \prime \prime}_{n,\, 1} \rangle^4 \big] \big)^{\frac{3+\delta}{4}}\, =\, O\Big( (m_{n}\, u_{n})^{\frac{3+\delta}{2}} \Big)
\end{equation*}
for a small $\delta > 0$. Now, combining this result with \eqref{TaylorExpRem} and H\"older's inequality, for sufficiently large values of $n$,
\begin{eqnarray}
\label{CvgTerm1}
\frac{t_{n}\, b_{n}}{n^\frac{3}{2}}\, \dE\big[ \vert \langle \lambda,\, \xi^{\, \prime \prime}_{n,\, 1} \rangle^3\, \de^{\, \nu_{n}} \vert \big] & \leq & \frac{t_{n}\, b_{n}}{n^\frac{3}{2}}\, \big( \dE\big[ \vert \langle \lambda,\, \xi^{\, \prime \prime}_{n,\, 1} \rangle \vert^{3+\delta} \big] \big)^{\frac{3}{3+\delta}}\, \big( \dE\big[ \de^{\frac{3+\delta}{\delta}\, \nu_{n}} \big] \big)^{\frac{\delta}{3+\delta}} \nonumber \\
 & \leq & \cst\, \frac{t_{n}\, b_{n}}{n^\frac{3}{2}}\, (m_{n}\, u_{n})^{\frac{3}{2}} ~ \longrightarrow ~ 0
\end{eqnarray}
by \eqref{TaylorExpRem}, \eqref{PropTUM} and the properties in \eqref{PropUn}. The second-order term in \eqref{TaylorCGF} satisfies
\begin{eqnarray}
\label{CvgTerm2}
\frac{t_{n}}{2\, n}\, \dE\big[ \langle \lambda,\,  \xi^{\, \prime \prime}_{n,\, 1} \rangle^2 \big] ~ = ~ \frac{t_{n}}{2\, n}\, \lambda\tr\, \dV(\xi^{\, \prime \prime}_{n,\, 1}) \lambda & = & \frac{t_{n}\, u_{n}\, m_{n}}{2\, n\, \nm B_{n}^{-1} \nm^3}\, \lambda\tr\, B_{n}^{-1}\, \dV(\vec(\zeta_{n,\, 1})) (B_{n}^{-1})\tr \lambda \nonumber \\
 & = & \frac{t_{n}\, u_{n}\, m_{n}}{2\, n}\, \lambda\tr\, \frac{B_{n}^{-1}\, \Upsilon_{n}\, (B_{n}^{-1})\tr}{\nm B_{n}^{-1} \nm^{\, 3}}\, \lambda \nonumber \\
 & \longrightarrow & \frac{1}{2}\, \langle \lambda, \Upsilon\, \lambda \rangle
\end{eqnarray}
where we used \eqref{PropTUM} and the results of Lemma \ref{LemVarTrunc}. The combination of \eqref{TaylorCGF}, \eqref{CvgTerm1} and \eqref{CvgTerm2} together with the G\"artner-Ellis theorem (see \textit{e.g.} Sec. 2.3 of \cite{DemboZeitouni98}) shows that the sequence
\begin{equation*}
\left( \frac{1}{b_{n}\, \sqrt{n}} \sum_{\ell=1}^{n} \xi_{n,\, \ell} \right)_{\!n\, \geq\, 1}
\end{equation*}
satisfies an LDP with speed $(b_{n}^{\, 2})$ and rate function given by the Fenchel-Legendre transform of the above logarithmic moment generating function, \textit{i.e.}
\begin{equation*}
I(x) = \sup_{\lambda\, \in\, \dR^{p^2}} \Big\{ \langle \lambda, x \rangle - \frac{1}{2}\, \langle \lambda, \Upsilon\, \lambda \rangle \Big\}.
\end{equation*}
Note that, due to its particular structure, $\Upsilon$ is only non-negative definite as soon as $p > 1$ (by way of example, its last row and column are zero). In that case (see \textit{e.g.} Ex. 1.1.4 of \cite{HiriartLemarechal12}, page 212), the explicit expression of this quadratic rate function, strictly convex on its relative interior, is
\begin{equation*}
I(x) = \left\{
\begin{array}{ll}
\frac{1}{2}\, \langle x, \Upsilon^{\, \dagger}\, x \rangle & \mbox{for } x \in \im(\Upsilon) \\
+\infty & \mbox{otherwise.}
\end{array}
\right.
\end{equation*}
After the truncation introduced in \eqref{PsiTrunc}, the decomposition \eqref{DecompVar} can be rewritten as
\begin{eqnarray*}
\frac{\sqrt{n}\, (1 - \rho(A_{n}))^{\frac{3}{2}}}{b_{n}}\, \vec\bigg( \frac{1}{n} \sum_{k=1}^{n} (\Phi_{n,\, k}\, \Phi_{n,\, k}\tr - \Gamma_{n}) \bigg) & = & \frac{(1 - \rho(A_{n}))^{\frac{3}{2}}\, \nm B_{n}^{-1} \nm^{\frac{3}{2}}}{b_{n}\, \sqrt{n}}\, \sum_{k=1}^{n} \xi_{n,\, k} \\
 & & \hsp +~ \frac{\sqrt{n}\, (1 - \rho(A_{n}))^{\frac{3}{2}}}{b_{n}}\, R^{\, *}_{n}
\end{eqnarray*}
where, in the remainder term $R^{\, *}_{n} = B_{n}^{-1}\, \vec(\Lambda_{n}) - R_{n}$, the residual of the truncation is given in \eqref{ResTrunc} and the main residual $R_{n}$ is given in \eqref{Res}. Lemma \ref{LemExpRes} and Lemma \ref{LemExpResTrunc} show that the first term in the right-hand is an exponentially good approximation of the left-hand side and that, as a consequence, they share the same LDP (see Def. 4.2.10 and Thm. 4.2.13 of \cite{DemboZeitouni98}). The contraction principle (see Thm. 4.2.1 of \cite{DemboZeitouni98}) enables to compute the rate function associated with the LDP, namely
\begin{equation}
\label{RateLDPCov}
I_{\Gamma}(x) = I\big( h^{-\frac{3}{2}}\, x \big) = \left\{
\begin{array}{ll}
\frac{1}{2\, h^3}\, \langle x, \Upsilon^{\, \dagger}\, x \rangle & \mbox{for } x \in \im(\Upsilon) \\
+\infty & \mbox{otherwise}
\end{array}
\right.
\end{equation}
where the limiting value $h > 0$ comes from \ref{HypRC}.
\hfill \qed

\subsubsection{Proof of Corollary \ref{CorMDPOLS}}
\label{SecProMDPOLS}

Using \eqref{OLSPen} and \eqref{VarPen},
\begin{eqnarray*}
\frac{\sqrt{n}}{b_{n}\, (1 - \rho(A_{n}))^{\frac{1}{2}}}\, \big( \whn^{\, \pi} - \theta_{n}^{\, \pi} \big) & = & \frac{\sqrt{n}\, (S_{n-1}^{\, \pi})^{-1}}{b_{n}\, (1 - \rho(A_{n}))^{\frac{1}{2}}}\, \sum_{k=1}^{n} \Phi_{n,\, k-1}\, \veps_{k} \\
 & = & \frac{n\, \nm B_{n}^{-1} \nm\, (S_{n-1}^{\, \pi})^{-1}}{b_{n}\, \sqrt{n}\, \nm B_{n}^{-1} \nm^{\frac{1}{2}}\, (1 - \rho(A_{n}))^{\frac{1}{2}}\, \nm B_{n}^{-1} \nm^{\frac{1}{2}}} \sum_{k=1}^{n} \Phi_{n,\, k-1}\, \veps_{k}.
\end{eqnarray*}
Our objective is first to prove that, for all $r > 0$,
\begin{equation}
\label{MDPOLS_CvgS}
\limn \frac{1}{b_{n}^{\, 2}} \ln \dP\bigg( \nm n\, \nm B_{n}^{-1} \nm\, (S_{n-1}^{\, \pi})^{-1} - \Gamma_{\pi}^{-1} \nm \geq r \bigg) = -\infty
\end{equation}
where $\Gamma_{\pi}$ is the invertible penalized variance \eqref{VarPen}, and then to establish an LDP for the sequence
\begin{equation}
\label{MDPOLS_Seq}
\left( \frac{1}{b_{n}\, \sqrt{n}\, \nm B_{n}^{-1} \nm^{\frac{1}{2}}} \sum_{k=1}^{n} \Phi_{n,\, k-1}\, \veps_{k} \right)_{\!n\, \geq\, 1}
\end{equation}
in order to obtain the announced result, \textit{via} the contraction principle (Thm. 4.2.1 of \cite{DemboZeitouni98}). On the one hand, we know from Theorem \ref{ThmMDPCov} and \eqref{RateLDPCov} that
\begin{eqnarray*}
\frac{1}{b_{n}^{\, 2}} \ln \dP\bigg( \nmg \frac{S_{n-1}}{n\, \nm B_{n}^{-1} \nm} - \frac{\Gamma_{n}}{\nm B_{n}^{-1} \nm} \nmg \geq r \bigg) & = & \frac{1}{b_{n}^{\, 2}} \ln \dP\bigg( \frac{\sqrt{n}}{b_{n}\, \nm B_{n}^{-1} \nm^{\frac{3}{2}}}\, \nmg \frac{S_{n-1}}{n} - \Gamma_{n} \nmg \geq r_{n} \bigg) \\
 & \longrightarrow & -\infty ~ = ~ -\lim_{\Vert x \Vert\, \rightarrow\, +\infty} I_{\Gamma}(x)
\end{eqnarray*}
since, by \ref{HypRC} and \ref{HypMD},
\begin{equation*}
r_{n} = \frac{r\, \sqrt{n}}{b_{n}\, \nm B_{n}^{-1} \nm^{\frac{1}{2}}} ~ \longrightarrow ~ +\infty
\end{equation*}
and $(1 - \rho(A_{n}))^{\frac{3}{2}} \sim h^{\frac{3}{2}}\, \nm B_{n}^{-1} \nm^{-\frac{3}{2}}$. So,
\begin{equation*}
\limn \frac{1}{b_{n}^{\, 2}} \ln \dP\bigg( \nmg \frac{S_{n-1}^{\, \pi}}{n\, \nm B_{n}^{-1} \nm} - \Gamma^{\, \pi}_{n} \nmg \geq r \bigg) = -\infty \hsp \text{for} \hsp \Gamma^{\, \pi}_{n} = \frac{\Gamma_{n}}{\nm B_{n}^{-1} \nm} + \pi\, I_{p}.
\end{equation*}
It is also clear that
\begin{equation*}
\bigg\{ \nmg \frac{S_{n-1}^{\, \pi}}{n\, \nm B_{n}^{-1} \nm} - \Gamma_{\pi} \nmg \geq r \bigg\} ~ \subset ~ \bigg\{ \nmg \frac{S_{n-1}^{\, \pi}}{n\, \nm B_{n}^{-1} \nm} - \Gamma^{\, \pi}_{n} \nmg \geq \frac{r}{2} \bigg\} \cup \bigg\{ \nm \Gamma^{\, \pi}_{n} - \Gamma_{\pi} \nm \geq \frac{r}{2} \bigg\}
\end{equation*}
and \eqref{LimGamGen} shows that the second event in the right-hand side becomes impossible when $n$ increases. Hence, from the reasoning above,
\begin{equation*}
\limn \frac{1}{b_{n}^{\, 2}} \ln \dP\bigg( \nmg \frac{S_{n-1}^{\, \pi}}{n\, \nm B_{n}^{-1} \nm} - \Gamma_{\pi} \nmg \geq r \bigg) = -\infty.
\end{equation*}
Now we shall use Lem. 2 of \cite{Worms99} to get \eqref{MDPOLS_CvgS}.

\smallskip

On the other hand, all the work consisting in proving that the sequence \eqref{MDPOLS_Seq} satisfies an LDP with speed $(b_{n}^{\, 2})$ has already been done in the proof of Theorem \ref{ThmMDPCov}. Indeed, \textit{via} the truncation \eqref{PsiTrunc},
\begin{eqnarray*}
\frac{1}{b_{n}\, \sqrt{n}\, \nm B_{n}^{-1} \nm^{\frac{1}{2}}} \sum_{k=1}^{n} \Psi_{n,\, k-1}\, \veps_{k} & = & \frac{1}{b_{n}\, \sqrt{n}\, \nm B_{n}^{-1} \nm^{\frac{1}{2}}} \sum_{k=1}^{n} \sum_{\ell=0}^{m_{n}-2} A_{n}^{\ell}\, E_{k-\ell-1}\, \veps_{k} \\
 & = & \frac{1}{b_{n}\, \sqrt{n}} \sum_{k=1}^{n} Z_{n,\, k}
\end{eqnarray*}
where the process $(Z_{n,\, k})_{k}$ forms a strictly stationary and $m_{n}$--dependent sequence. However, apart from the renormalization, this is precisely the first column of the first term of \eqref{ZetaTrunc}. Thus, the calculations are similar and we find, like in Lemma \ref{LemVarTrunc},
\begin{equation*}
\dV(Z_{n,\, 1}) = \frac{\sigma^2\, \Gamma_{n,\, m_{n}}}{\nm B_{n}^{-1} \nm}.
\end{equation*}
In that case, from the convergence \eqref{LimGam} and the previous proof, the rate function associated with the LDP is given by
\begin{equation*}
J(x) = \sup_{\lambda\, \in\, \dR^{p}} \Big\{ \langle \lambda, x \rangle - \frac{\sigma^2}{2}\, \langle \lambda, \Gamma\, \lambda \rangle \Big\} = \left\{
\begin{array}{ll}
\frac{1}{2\, \sigma^2}\, \langle x, \Gamma^{\, \dagger}\, x \rangle & \mbox{for } x \in \im(\Gamma) \\
+\infty & \mbox{otherwise.}
\end{array}
\right.
\end{equation*}
The exponential negligibility of the remainder of the truncation is obtained by following the lines of Lemma \ref{LemExpResTrunc}. The contraction principle enables to compute the rate function associated with the LDP, namely
\begin{equation}
\label{RateLDPOLS}
I_{\theta}(x) = J\big( \Gamma_{\pi} \sqrt{h}\, x \big) = \left\{
\begin{array}{ll}
\frac{h}{2\, \sigma^2}\, \langle x, \Gamma_{\pi}\, \Gamma^{\, \dagger}\, \Gamma_{\pi}\, x \rangle & \mbox{for } x \in \im(\Gamma_{\pi}^{-1}\, \Gamma) \\
+\infty & \mbox{otherwise}
\end{array}
\right.
\end{equation}
where the exponential convergence \eqref{MDPOLS_CvgS} has been combined to the LDP established on the sequence \eqref{MDPOLS_Seq}.
\hfill \qed

\smallskip

\noindent \textbf{Acknowledgements}. The author thanks the associate editor and the two anonymous reviewers for the numerous comments and suggestions that clearly helped to improve the paper. He also thanks R. Garbit for the constructive discussion about the link between Vandermonde matrices and Lagrange polynomials.

\nocite{*}

\bibliographystyle{acm}
\bibliography{MDP_NearlyUnstable}

\end{document}